\documentclass[11pt,reqno]{amsart}
\usepackage{amsmath}
\usepackage{amsthm}
\usepackage{graphicx}
\usepackage{textgreek}
\usepackage{upgreek}
\usepackage{hyperref}
\usepackage{mathrsfs}
\usepackage{color}
\usepackage{bm}
\usepackage{amsfonts}
\usepackage{amssymb}
\usepackage{csquotes}
\usepackage{enumerate}

\usepackage{soul}
\usepackage[margin=1.1in]{geometry}
\numberwithin{equation}{section}

\newtheorem{theorem}{Theorem}[section]

\theoremstyle{definition}

\newtheorem{remark}[theorem]{Remark}

\theoremstyle{remarks}

\newtheorem{thm}{Theorem}
\newtheorem{lem}[thm]{Lemma}
\newtheorem{prop}[thm]{Proposition}

\newtheorem{defi}[thm]{Definition}

\newtheorem{rem}[thm]{Remark}
\newtheorem{sch}[thm]{Scheme}

\def\RR{{\mathbb R}}

\title{Selection of equilibria in a linear quadratic mean-field game}

\author{Fran\c{c}ois Delarue, Rinel Foguen Tchuendom}

\begin{document}
\maketitle



\begin{abstract}
In this paper, we address an instance of uniquely solvable mean-field game with a common noise whose corresponding counterpart without common noise has several equilibria. We study the selection problem for this mean-field game without common noise via three approaches. 

A common approach is to select, amongst all the equilibria, those yielding the minimal cost for the representative player. Another one is to select equilibria that are included in the support of the zero noise limit of the mean-field game with common noise. A last one is to select equilibria supported by the limit of the mean-field component of the corresponding $N$-player game as the number of players goes to infinity. The contribution of this paper is to show that, for the class under study, the last two approaches select the same equilibria, but the first approach selects another one.  
\vskip 4pt


\vskip 4pt


\end{abstract}

\section{Introduction}

\subsection{A short overview of MFGs}

The theory of Mean Field Games (MFGs for short) is concerned with the study of asymptotic Nash equilibria for stochastic differential games with an infinite number of players subject to a mean-field interaction (i.e each player is affected by the other players only through the empirical distribution of the system).
In this regard, it is worth recalling that a Nash equilibrium constitutes a consensus (or compromise) between all the players from which no player has unilateral incentive to deviate.

As the number of players (which we denote by the upper case $N$ throughout the paper) of the stochastic differential game increases, finding Nash equilibria becomes an increasingly complex problem as it typically involves a system of $N$ PDEs set on a space of dimension of order $N$. The motivation for studying the asymptotic regime is to reduce the underlying complexity. At least in the case where the players are driven by independent noises, the hope is indeed to take benefit from the theory of propagation of chaos for mean-field interacting systems (see for example \cite{Sznitman}) in order to reduce the analysis of the whole system to the analysis of a single representative player. 

In the analysis of the limiting MFGs, the representative player aims at minimizing a cost functional while interacting with an environment described by a flow of distributions. Finding Nash equilibria thus consists in finding optimal states whose flow of marginal distributions matches exactly the flow of distributions describing the environment. This is a constraint of McKean-Vlasov type which requires to solve a fixed point problem over the set of time-dependent paths with values in the space of probability measures. 

MFGs were introduced independently and simultaneously by Lasry and Lions \cite{MFG1,MFG2,MFG3} and  by Caines, Huang and Malham\'e \cite{HuangCainesMalhame2} (who used the name of Nash Certainty Equivalence). We refer to the notes \cite{Cardaliaguet} written by Cardaliaguet for a very good introduction to the subject. We also refer to the works of Carmona and Delarue, who studied MFGs with a probabilistic approach, see for instance \cite{CarmonaDelarue_sicon,CarmonaDelarueLachapelle} together with the recent two-volume monograph \cite{CarmonaDelarue_book_I,CarmonaDelarue_book_II}. Many other authors have contributed to the rapid development of the theory, see the references in \cite{CarmonaDelarue_book_I,CarmonaDelarue_book_II}. Under suitable regularity conditions of the cost functional, existence of Nash equilibria has been proved in the above works (for instance by using Schauder's fixed point argument). Further  monotonicity conditions introduced by Lasry and Lions guarantee uniqueness of the solution, see \cite{MFG3}.

\subsection{A specific class of MFGs}
In the earlier paper \cite{Foguen}, Foguen Tchuendom investigated a class of Linear-Quadratic Mean Field Games (LQ-MFGs) in which the representative player at equilibrium interacts with the mean of its distribution. Here and below, we call LQ-MFG a mean-field game whose cost functionals are quadratic in the state and control variables and whose dynamics is linear in the state and control variables: \textit{Still, the coefficients may depend in 
a more general fashion upon the distribution of the population; this is contrast with earlier works on mean-field games, in which 
the coefficients of LQ-MFGs are also required to be linear or quadratic with respect to the mean of the population}.

In \cite{Foguen}, the $N$ players in the finite game are also assumed to be subject to a common (or systemic) noise in addition to independent noises. Such a modeling is motivated by practical applications.  
We refer to 
the review of 
Gu\'eant, Lasry and Lions \cite{GueantLasryLions.pplnm} for earlier examples of mean-field games 
with a common noise. 
We also refer to the linear-quadratic model (including linear-quadratic coefficients with respect to the mean 
of the population) introduced by Carmona, Fouque and Sun
\cite{CarmonaFouqueSun} for another example involving a common noise. In comparison with 
mean-field games without common noise, the major change in mean-field games with common noise
is that, due to the presence of common noise, the representative player at equilibrium feels the mean-field interaction through the conditional distribution of its state given the common noise. 
In \cite{Foguen}, equilibria to the LQ-MFGs with common noise under study are shown to be characterized by a one dimensional standard 
Forward-Backward Stochastic Differential Equation
(FBSDE for short). In this FBSDE, the forward part describes the conditional mean of the representative player' 
state given the common noise and the backward one accounts for the affine part of the feedback control. 
Thanks to the common noise, this FBSDE is non-degenerate and thus satisfies an existence and uniqueness theorem proved by Delarue in \cite{Delarue02}; in particular, the LQ-MFGs with common noise 
addressed in \cite{Foguen}
have a unique equilibrium. Importantly, 
\cite{Foguen} provides a counter-example to uniqueness of Nash equilibria for a mean-field game in the same class of LQ-MFGs but in the absence of common noise.

\subsection{From restoration of uniqueness to selection of equilibria}
\label{subse:1:3}

The result obtained in \cite{Foguen} is an example of restoration of uniqueness by addition of a common noise. The striking fact in this example is that the action of the common noise onto uniqueness is limpid. Basically, 
the LQ structure forces equilibria to be (one-dimensional) Gaussian processes conditional on the realization of the common noise: Whilst the covariance structure is independent of the realization of the common noise (and hence is deterministic), the conditional mean follows an Ornstein-Ulhenbeck process driven by the common noise. Hence, the class of LQ-MFG's tackled in \cite{Foguen} is parametric (the parameter being one-dimensional) and the role of the common noise is precisely to force the parameter randomly. The fact that equilibria have a one-dimensional structure plays a crucial role in the rest of the paper.   
In this regard, it is worth mentioning that  
Foguen Tchuendom's result was extended by Delarue in \cite{DelarueRestoration} to a larger class of nonparametric non-LQ-MFGs but at the cost of using an infinite dimensional common noise, which makes it much too complex for our purpose.

Foguen Tchuendom's result prompts us to address the following question: For an instance of LQ-MFG which is uniquely solvable under the presence of common noise but which has several equilibria in the absence of common noise, is there any way to select some specific equilibria to the game without common noise? 
To answer this question, we propose here three methods of selection, as described below:

\begin{enumerate}
\item (minimal cost selection) Amongst all the equilibria to the LQ-MFG without common noise, select those that minimize the cost of the representative player;
\item (zero noise limit selection) Consider the unique stochastic equilibrium to the LQ-MFG with common noise and find its weak limit as the intensity of the noise tends to $0$. If this limit exists, select the equilibria that are included in the support of the limit. 
\item ($N$-player limit selection) Solve the $N$-player game without common noise and find the weak limit 
of the equilibrium as $ N \rightarrow +\infty $. If this limit exists, select the equilibria that are included in the support of the limit.
\end{enumerate}

Whilst the first method is directly connected with the optimization structure underpinning the LQ-MFG, the second approach is in fact much more general. Indeed, the idea of restoring uniqueness by means of a random forcing has been extensively studied in probability theory. It goes back to the earlier work of Zvonkin \cite{Zvonkin} on the solvability of one-dimensional stochastic differential equations driven by non-Lipschitz continuous drifts. Several authors also contributed to the subject and addressed the higher dimensional framework, among which Veretennikov \cite{veretennikov}, Krylov and R\"ockner 
\cite{kry:roc:05}, Davie \cite{dav:07}... 
Similar questions have been also addressed in the framework of infinite dimensional stochastic differential equations, see for instance Flandoli, Gubinelli and Priola \cite{FlandoliGubinelliPriola} and the monograph of Flandoli \cite{Flandoli}. 
Still, although restoration of uniqueness has been investigated in various 
frameworks, including, as we just mentioned, infinite dimensional ones, 
finding the zero-noise limit when the corresponding deterministic or ordinary equation has multiple solutions
is a challenging question, for which fewer results are known. The earlier result in this direction is due to 
Bafico and Baldi, see \cite{BaficoBaldi}; it provides a rather complete picture of the selection procedure for one-dimensional dynamics with isolated singularities. 
Examples treated in \cite{BaficoBaldi} will serve us as a benchmark throughout this paper, but, in fact, we will mostly follow 
another approach
to these examples  due to 
Delarue and Flandoli \cite{DelarueFlandoli}. 
In \cite{DelarueFlandoli}, the authors not only address the zero noise limit but also 
make explicit the typical time at which selection occurs; our strategy is to do the same below. 
We refer the reader to 
\cite{Trevisan} for a third proof, to 
\cite{HerrmannRoynette} for related large deviations principles, 
and to 
\cite{AttanasioFlandoli,DelarueFlandoli2,JourdainReygner} for examples of selection in higher (possibly infinite) dimension.

As for the third method of selection, it is directly connected with the fact that mean-field games are understood as limiting versions of games with finitely many players. In this regard, it is a rather challenging question to show that equilibria to 
the finite player games do converge to a solution of the corresponding mean-field game. In fact, the approach to this question depends on the nature of the equilibria: In the finite player system, equilibria may be searched in an open or closed loop form. 
As for open loop equilibria, as considered in \cite{Foguen}, weak compactness methods were first studied by 
Fischer \cite{Fischer}
and
Lacker
\cite{Lacker_limits}. Generally speaking, the point therein is to prove that the support of any weak limit of the laws of the empirical distributions of 
the finite player game equilibria is included in the set of solutions to the limiting mean-field game. Still, to the best of our knowledge, nothing has been said so far on the exact shape of this support (at least when there is no uniqueness): This is the question we want to address below in the particular example specified in the next section. 
As for equilibria in closed loop form, the first main general result on the convergence of equilibria, at least when uniqueness to the limiting mean-field game holds true, is due to 
Cardaliaguet, Delarue, Lasry and Lions \cite{CardaliaguetDelarueLasryLions} and is based on the so-called master equation 
for mean-field games, which is a PDE set on the space of probability measures: As our case is parametric, the master equation for it reduces (up to a correction term) to a one-dimensional standard PDE. 
Importantly, this PDE here takes the form of a scalar conservation law and, although uniqueness 
does not hold for the mean-field game under study and, accordingly, the master equation does not admit
a classical solution, the so-called entropy solution of this scalar conservation law
is intended to be, amongst all the possible solutions, of a special interest. And, indeed, although 
we deal with open loop equilibria, we make an intense use of it throughout the text.
We will go back to this point next. In fact, the reader must be aware that, very recently, Lacker \cite{Lacker2} succeeded to 
extend the weak compactness approach initiated in \cite{Lacker_limits} to closed loop equilibria in order to tackle cases when uniqueness does not hold and henceforth when the master equation is ill-posed. 
Although this new result in the literature on mean-field games is not of a special use in the sequel, it demonstrates that, similar to the question we here address for open loop equilibria, the identification of the weak limits of closed loop equilibria in case when uniqueness does not hold is a hot question as well. In this regard, it 
is worth mentioning that, in a parallel forthcoming work to ours, Cecchin, Dai Pra, Fisher
and Pelino \cite{Cecchin} address a similar question but for equilibria in closed loop form: Namely, 
for a continuous time mean-field game on a two-state space for which uniqueness does not hold,  
they investigate the equilibria that are selected by the limit of the  closed-loop equilibria of the $N$-player game.

\subsection{Summary of the results and organization of the paper}
Generally speaking, we focus below on a specific example of the aforementioned parametric LQ-MFG class, which admits three equilibria.
As made clear in the text, these three equilibria are parametrized by the three following values of the parameter: $-1$, $0$ and $1$. In this framework, we prove that
the zero noise limit and $N$-player limit selection methods select 
with probability $1/2$ the two equilibria $-1$ and $1$ whilst the minimal cost selection 
selects the equilibrium $0$. 
So, the first striking fact of this paper is to show that the minimal cost selection does not yield the same result as the other two approaches! The second one is to show that, here, taking the vanishing viscosity limit and
taking the limit over the number of players give the same result: Intuitively, 
the idiosyncratic noises in the $N$-payer game here aggregate into a common noise of intensity 
of order $1/\sqrt{N}$, which explains why the two approaches yield the same result.  
The last important point is that this example shows that several equilibria may be physically selected in this way; 
even more, it makes clear the fact that randomized equilibria may naturally appear for mean-field games
without common noise: Using the same terminology as in 
\cite{Lacker2} (see also 
\cite{CarmonaDelarue_book_II,CarmonaDelarueLacker,Lacker_limits}), the limit that picks up the two equilibria $-1$ and $1$ with probability $1/2$
should be regarded as a weak mean-field equilibrium; also, the randomness that carries 
the choice between $-1$ and $1$ should be regarded as an endogenous common noise in a game without exogenous common noise!  
\vskip 4pt

Our result should be compared with \cite{Cecchin}
and \cite{Lacker2}. 
First, it is worth mentioning that our work has some similarities with \cite{Cecchin}: As in 
\cite{Cecchin}, we show that the master equation for the LQ-MFG without uniqueness can be addressed by using the theory of entropy solutions to scalar conservation laws; 
as we already mentioned,
there is a (one-dimensional) nonlinear hyperbolic equation underpinning the master equation
and its entropy solution 
permits to identify 
the optimal feedback that is selected by both the zero noise and $N$-player limits.
For sure, this 
general fact 
certainly goes beyond the two examples tackled here and in 
\cite{Cecchin}. Certainly, it
should be addressed in a more systematic way in the future, at least in  \textit{one-dimensional parametric} models.
Indeed, the key point in both papers is that equilibria are driven by a one-dimensional parameter:
Here, the parameter is the mean of the one-dimensional state variable and, in 
\cite{Cecchin}, it is the probability weight of one of the two elements of the state space. 
Although it sounds to be a very exciting question, selection in higher (but finite) dimensional parametric model is probably much more challenging:
The master equation is then expected to reduce to a more complicated non-conservative hyperbolic system. 
The latter fact is made clear in \cite[Section 2]{BLL} for mean-field games with a finite state space of any arbitrary cardinality. Similarly, the model 
we address below can be also written out in higher dimension $d \geq 2$, but, then, the aforementioned one-dimensional PDE
underpinning the master equation turns into a non-conservative hyperbolic system as well.
 
 Another interesting remark about \cite{Cecchin} is that, in the example addressed therein, the 
$N$-player limit selects one equilibrium only while it selects two equilibria in our example: The difference comes from the fact that we here choose an initial condition that exactly seats at the singularity of the entropy solution of the conservation law. As explained below, in our framework, our method can be adapted to handle initial conditions that are away from the singularity, in which case one equilibrium only is selected by the 
zero-noise and $N$-player limits.

Lastly, it must be stressed that a related question to ours is studied in \cite[Subsection 7.2]{Lacker2}:
For a pretty similar LQ-MFG, it is proven in 
\cite{Lacker2}
that, given the weak mean-field equilibrium that charges, with symmetric weights, the two equilibria $-1$ and $1$ of the 
game without uniqueness, it is possible to construct a sequence of approximate Nash equilibria 
that converges to it in the weak sense. 
\vskip 4pt

The paper is organized as follows. We implement the first method, which we call {minimal cost selection}, 
in Section 
\ref{se:2}. In Section 
\ref{se:PDE}, we make clear what is the notion of master equation in our setting.
It plays a key role in the subsequent analysis of the zero-noise limit and of the convergence of
the $N$-player equilibria. 
Section \ref{se:zero:noise} 
is dedicated to the analysis of the zero-noise limit, 
whilst we focus on the limit of the $N$-player equilibria in 
Section \ref{se:zero:noise}. Further computations, that are used in the text, are detailed in Appendix.

\section{Notations and statements}
\label{se:2}
\subsection{Description of the mean-field game and related selection of equilibria}
For the sake of clarity, we recall the class of LQ-MFGs addressed in \cite{Foguen} and the corresponding characterization of equilibria through FBSDEs. 

We are given two independent (one-dimensional) Brownian motions $ B = (B_t)_{t \in [0,T]}$ and $W=(W_t)_{t \in [0,T]}$ defined on a complete filtered probability space
$(\Omega, \mathcal{F}, (\mathcal{F}_t)_{t\in [0,T]}, \mathbb{P})$ satisfying the usual conditions. The representative player's initial state is given in the form of a random variable $\xi \in \mathcal{L}^2_{\mathcal{F}_0}$, $\mathcal{L}^2_{\mathcal{F}_0}$ standing for the collection of square integrable $\mathcal{F}_0$-measurable random variables. We suppose (mostly for convenience) that the filtration  $(\mathcal{F}_t)_{t\in [0,T]}$ corresponds to the natural filtration generated by ${\xi, W, B}$ augmented with $\mathbb{P}$-null sets. Also, we let $(\mathcal{F}^B_t)_{t\in [0,T]}$ be the filtration generated by $ B$ only and augmented with $\mathbb{P}$-null sets. 

Throughout the paper, 
we consider controls $\alpha := (\alpha_t)_{t\in[0,T]} \in \mathcal{H}^2$, where $\mathcal{H}^2$ is the space of $ (\mathcal{F}_t)_{t\in [0,T]}$-progressively measurable processes satisfying 
$$ \mathbb{E}\Big[ \int_0^T|\alpha_t|^2 dt\Big] < +\infty.$$ Finally, we consider three constants $\kappa \in \mathbb{R},\sigma \geq 0, \sigma_{0} \geq 0$ and three bounded and Lipschitz continuous functions $ f, b, g : \mathbb{R} \rightarrow \mathbb{R} $.  The MFG problem considered in \cite{Foguen} reads:

\begin{sch}{(MFG-problem)}
\label{MFG-problem}

\begin{enumerate} 

\item (Mean field Input) If $\sigma_{0} >0$, consider a continuous $(\mathcal{F}^B_t)_{t \in [0,T]}$-adapted process $(\mu_t)_{t \in [0,T]}$ taking values in $\mathbb{R}$. 
If $\sigma_{0}=0$, take $(\mu_{t})_{t \in [0,T]}$ as a deterministic (continuous real-valued) curve. 

\item (Cost Minimization) \textit{Find} $\alpha^{*} \in \mathcal{H}^2$, satisfying
\begin{align*}
\label{costfxn2}
J(\alpha^*) = \min_{\alpha \in \mathcal{H}^2} J(\alpha)  & := \min_{\alpha \in \mathcal{H}^2}  \mathbb{E} \Bigg[ \int_0^T \frac{1}{2} [\alpha_t^2 + (f(\mu_t) + X_t)^2 ] dt  + \frac{1}{2}(X_T + g(\mu_T))^2 \Bigg] 
\end{align*}

\textit{under the stochastic dynamics}
\begin{equation}
\label{constraint2}
\begin{cases}
dX_t = [ \kappa X_t + \alpha_t + b(\mu_t)]dt + \sigma dW_t + \sigma_0 dB_t , \hspace{2mm} \forall t \in [0,T],
\\ 
 X_0 = \xi. 
 \end{cases}
\end{equation}

\item (McKean-Vlasov constraint) If $\sigma_{0}>0$, find $(\mu_t)_{t \in [0,T]}$ such that: 
\begin{equation*} 
\forall t \in [0,T],\hspace{2mm} \mu_t = \mathbb{E}[X^{\alpha^*}_t| \mathcal{F}^B_T]. 
\end{equation*}
If $\sigma_{0} =0$, find $(\mu_t)_{t \in [0,T]}$ such that the above holds true without conditional expectation. 
\end{enumerate}
\end{sch}

We recall from \cite{Foguen} that one can characterize the solutions of this MFG-problem through FBSDEs as in the following proposition. 

\begin{prop}
\label{prop1}
Given the above data with $\sigma_{0}>0$, 
there exists an MFG-solution $(\alpha_t,\mu_t)_{t \in [0,T]} $  if and only if 
there exists an $({\mathcal F}_{t}^{B})_{t \in [0,T]}$ adapted solution
$(\mu^{\xi,\sigma_0}_t, h^{\xi,\sigma_0}_t, Z^{\xi,\sigma_0}_t)_{t \in [0,T]}$
 to the FBSDE:
\begin{equation}
\label{fbsde1}
\begin{cases}
\forall t \in [0,T],
\vspace{2pt}
\\ 
d\mu^{\xi,\sigma_0}_t = \bigl[ -w^{-2}_t h^{\xi,\sigma_0}_t + w^{-1}_t b(w_t \mu^{\xi,\sigma_0}_t) \bigr]dt + w^{-1}_t \sigma_0 dB_t,   
\vspace{2pt}
\\ 
dh^{\xi,\sigma_0}_t  = \bigl[ -w_t f(w_t \mu^{\xi,\sigma_0}_t) - w_t \eta_t b(w_t \mu^{\xi,\sigma_0}_t) \bigr]dt + Z^{\xi,\sigma_0}_t dB_t,\vspace{2pt}
\\ 
\textrm{\rm and} \quad 
\mu^{\xi,\sigma_0}_0 = \mathbb{E} [\xi]w^{-1}_0, \hspace{1mm} h^{\xi,\sigma_0}_T =g(\mu^{\xi,\sigma_0}_T),
\end{cases}
\end{equation}
where
\begin{align*}
&w_t  :=  \exp \Big( \int_t^T (- \kappa+ \eta_s)ds \Big), \quad \forall t \in [0,T],
 \notag 
\\
&\eta := (\eta_t)_{t\in[0,T]}  \hspace{2mm} \text{is the unique solution to the Riccati ODE:} \hspace{2mm} \frac{d \eta_t}{dt} = \eta^2_t -2\kappa \eta_t - 1, \hspace{2mm} \eta_T = 1. 
\end{align*}
When FBSDE (\ref{fbsde1}) is solvable, $(\alpha_{t},\mu_{t})_{t \in [0,T]}$
and $(\mu^{\xi,\sigma_0}_{t},h_{t}^{\xi,\sigma_{0}})_{t \in [0,T]}$
are connected by the following relationships:
\begin{equation*}
\begin{split}
&\mu_{t} = w_{t} \mu^{\xi,\sigma_0}_t, \quad \forall t \in [0,T],
\\
&\alpha_t =  -\eta_t X_t - h_t, 
\quad \textrm{\rm where} \quad h_{t} = w_{t}^{-1} h^{\xi,\sigma_0}_t, \quad \forall t \in [0,T],
\end{split}
\end{equation*}
$(X_{t})_{t \in [0,T]}$ being implicitly defined as the solution of the 
forward equation (\ref{constraint2}).

The result remains true when $\sigma_{0}=0$ except that 
$Z^{\xi,\sigma_0}$ in 
\eqref{fbsde1}
is null, that is to say 
\eqref{fbsde1} is a deterministic system. 
\end{prop}

We know from \cite{Foguen} that  
\begin{enumerate}
\item In the presence of common noise (i.e $\sigma_0 > 0$), FBSDE (\ref{fbsde1}) is 
uniquely solvable, in which case there is a unique equilibrium to the LQ-MFG;
\item In the absence of common noise  (i.e $\sigma_0 = 0$),
FBSDE (\ref{fbsde1}) is solvable, but it may admit several solutions. In that case, which we call \textit{degenerate}, 
there may be several equilibria.
\end{enumerate}

\subsection{A particular case}
\label{subse:particular:case}
Throughout the paper, we consider the particular case when $ f = b = \xi = 0$  
and $g : \mathbb{R} \rightarrow \mathbb{R}$ given by  
\begin{equation}
\label{eq:g}
 g(x) :=  - \frac{x}{r_\delta}  \mathbf{1}_{ |x| \leq r_\delta } - \textrm{\rm sign}(x) \mathbf{1}_{ |x| > r_\delta },
 \end{equation}
where for  a fixed time $\delta \in (0,T)$,  $r_\delta := \int_{\delta}^T  w^{-2}_s ds > 0.$ 

Proposition \ref{prop1} states that, in order to find an equilibrium to this particular LQ-MFG, it is sufficient (and in fact necessary as well) to find a continuous, $({\mathcal F}_{t}^{B})_{t \in [0,T]}$ adapted, solution $(\mu^{0,\sigma_0}_t, h^{0,\sigma_0}_t, Z^{0,\sigma_0}_t  )_{t\in[0,T]}$ to the FBSDE  
 
\begin{equation}
\label{fbsde2}
\begin{cases}
d\mu^{0,\sigma_0}_t = -w^{-2}_t h^{0,\sigma_0}_t dt + w^{-1}_t \sigma_0 dB_t, \hspace{2mm} \forall t \in [0,T],
\\ 
dh^{0,\sigma_0}_t  =  Z^{0,\sigma_0}_t dB_t, \hspace{2mm} \forall t \in [0,T], 
\\ 
\mu^{0,\sigma_0}_0 = 0, \hspace{1mm} h^{0,\sigma_{0}}_T = g(\mu^{0,\sigma_0}_T) \hspace{1mm} .
\end{cases}
\end{equation}

In the presence of common noise (i.e $\sigma_0 > 0$), the FBSDE (\ref{fbsde2}) has a unique solution. Thus there exists a unique equilibrium  $(w_t \mu^{0,\sigma_0}_t, \alpha_t =  -\eta_t X_t - w_{t}^{-1} h^{0,\sigma_0}_t )_{t\in[0,T]}$ whose randomness depends only on the common noise $B$.

In absence of common noise (i.e $\sigma_0 = 0$), the system 
\eqref{fbsde2} becomes 
\begin{equation}
\label{fbsde2:sigma0=0}
\begin{cases}
d\mu^{0,0}_t =  -w^{-2}_t h^{0,0}_t dt, \hspace{2mm} \forall t \in [0,T],
\\ 
dh^{0,0}_t  = 0, \hspace{2mm} \forall t \in [0,T], 
\\ 
\mu^{0,0}_0 = 0, \hspace{1mm} h^*_T = g(\mu^{0,0}_T) \hspace{1mm} .
\end{cases}
\end{equation}

Our analysis is based upon the following observation that 
\eqref{fbsde2:sigma0=0}
has multiple solutions: 

\begin{prop}
\label{prop:sol:sigma0=0}
There exist three solutions to (\ref{fbsde2:sigma0=0}), which are 
\begin{equation}
\label{equilibria}
(\mu^{0,0}_t, h^{0,0}_t, Z^{0,0}_t  )_{t\in[0,T]} = \bigg(-A \int_0^t w^{-2}_s ds , A, 0  \bigg)_{t\in[0,T]} \quad \textrm{\rm for} \quad A \in \{ -1,0,1\}.
\end{equation}
\end{prop}

\begin{proof} The first point is to check that the functions given in the statement are indeed solutions to the equation. In fact, the only difficult point is to check the boundary condition. When $A=0$, there is no difficulty. When $A=1$, we observe that 
$\vert \mu^{0,0}_T \vert \geq r_{\delta}$. Hence, $g( \mu^{0,0}_T) = 1$, which is indeed equal to $1$. The case 
$A=-1$ is treated in the same way.

It then remains to check that there are no other solutions. In fact, whatever the solution, the process
$(h_{t}^{0,{0}})_{t \in [0,T]}$ must be constant, hence it must be equal to some $A \in \RR$. Then, 
$\mu^{0,{0}}_{T} = -A \int_{0}^T w^{-2}_{t} dt$. 

If $\vert A \int_{0}^T w^{-2}_{t} dt \vert \leq r_{\delta}$, then the terminal boundary 
condition writes $A = -A \int_{0}^T w^{-2}_{t} dt/r_{\delta}$, which yields $A=0$. 
If $\vert A \int_{0}^T w^{-2}_{t} dt \vert > r_{\delta}$, the boundary condition is in $\{-1,1\}$ and we get 
$A \in \{-1,1\}$.

\end{proof}

\subsection{Main statement}

Referring to the three approaches detailed in Subsection \ref{subse:1:3}, our main statement has the following form:

\begin{thm}
\label{thm:4}
As for the example introduced 
in Subsection \ref{subse:particular:case}, the minimal cost selection selects, in the regime $\sigma_{0}=0$, the
 equilibrium corresponding to $A=0$, whilst the zero-noise limit and the $N$-player game 
 (under the additional assumption that $\sigma >0$)
 approaches select
 a randomized equilibrium, as given by the equilibrium $A=1$ with probability $1/2$ and by the equilibrium $A=-1$ with probability $1/2$.  
\end{thm}

We refer to Subsection 
\ref{subse:Nplayer_game} for a clear meaning of what we call $N$-player game in the framework under study. 
\vskip 4pt

Although the rule of selection based upon minimal cost is sometimes met in the literature, this result shows that it leads in fact to contradictory results with the other rules of selection addressed in the paper.
In fact, this should not come as a surprise. Indeed, 
it is worth mentioning that, in Scheme
\ref{MFG-problem}, we can add any function of $\mu_{T}$ 
to the terminal cost entering the definition of 
$J$. Obviously, this should not change the minimizers of $J$ since the value of 
$\mu_{T}$ is kept frozen in the optimization procedure. Still, this may certainly modify the output 
of the minimal cost selection method. This strongly suggests that the minimal  cost 
selection method is of a limited scope. 
\vskip 4pt

Regarding the two other selection rules, we draw reader's attention to the following two points. 
First, the randomized equilibrium given by $A=\pm 1$ with probability $1/2$ should be regarded 
as an equilibrium on its own. It requires a modicum of care to write out the matching condition 
in item (3) of Scheme 
\ref{MFG-problem}, but it can be done at the price of conditioning on the value of $A$: Given the fact that $A=\pm 1$ is selected, 
the conditional mean of the state variable is 
$(-A \int_0^t w^{-2}_s ds)_{t\in[0,T]}$, see
\eqref{equilibria}. In other words, the randomized equilibrium carries an endogenous systemic noise.
Actually, this property is pretty similar to the one encountered for weak solutions to stochastic differential equations, 
see 
\cite{StroockVaradhan}, which carry an extra randomness in addition to the exogeneous noise driving the equation. By analogy, the randomized equilibrium could be called a weak mean-field equilibrium, 
see for instance 
\cite{CarmonaDelarue_book_II,CarmonaDelarueLacker,
Lacker2,Lacker_limits} for more details. In this regard, it is worth mentioning that 
endogenous noises also appear in the analysis of the Peano 
phenomenon by vanishing viscosity method, see 
\cite{BaficoBaldi}. 

 In fact, and this is the second point we want to stress, the zero-noise limit and $N$-player game approaches here select two equilibria (and not one equilibrium) because of the choice we made for the initial condition $\xi$. As we show below, $\xi = 0$ is indeed the discontinuity point of the entropy solution 
of a certain scalar conservation law that underpins the game, see 
\eqref{teta}. In this regard, it is the worst (meaning the most unstable) initial point 
that we can guess. 
In fact, we could have chosen an initial condition $\xi \not =0$. In this framework, 
we have an extension of Proposition 
\ref{prop:sol:sigma0=0}:
 Whenever
$\vert \xi \vert < \int_{0}^{\delta} w_{s}^{-2} ds$, the 
FBSDE \eqref{fbsde2:sigma0=0} with $\mu_{0} = \xi$ as initial condition has three solutions, which are
\begin{equation}
\label{equilibria}
(\mu_t, h_t, Z_t  )_{t\in[0,T]} = \bigg(\xi -A \int_0^t w^{-2}_s ds , A, 0  \bigg)_{t\in[0,T]} \quad \textrm{\rm for} \quad A \in \Bigl\{ -1,\frac{\xi}{\int_{0}^{\delta} w_{s}^{-2} ds},1 \Bigr\}.
\end{equation}
In this framework, 
the methodology we develop for addressing the case $\xi =0$ also applies 
to the case $\vert \xi \vert < \int_{0}^{\delta} w_{s}^{-2} ds$, $\xi \not =0$: 
The key tool is Proposition 
\ref{boundasymp}. It shows the following: When initiating the mean field 
game with a common noise of intensity $\sigma_{0}>0$ from $\xi \not =0$, the conditional mean of the representative
player 
stays away from $0$ with probability asymptotically equal to $1$ as $\sigma_{0}$ tends to $0$; 
similarly,  when initiating the $N$-player game from a common point $\xi \not =0$, the empirical mean of the $N$ players stays away from $0$ with probability asymptotically equal to $1$ as $N$ tends to $\infty$. 
In that case, both approaches should select only one equilibrium among the above three ones: 
When $\xi >0$, they should select $A=-1$ as it is the only one for which $(\mu_{t})_{t \in [0,T]}$
in 
\eqref{equilibria}
remains positive; similarly, when $\xi <0$, they should select $A=1$. 
For sure, the case $\xi = 0$ is more difficult as the first step is precisely to show that 
the conditional mean of the representative player in the mean-field game with common noise 
or the empirical mean of the players in the $N$-player game go sufficiently far away from 
$0$ before they stay either in the positive or negative half-plane: This is the so-called notion of 
transition point introduced in 
Subsection \ref{subse:transition:point} that makes this fact clear.

\section{Minimal cost selection}
\label{se:minimal}

Keep in mind the particular LQ-MFG (\ref{fbsde2}) and focus more specifically on the case without common noise 
(i.e $ \sigma_0 = 0$), see \eqref{fbsde2:sigma0=0}. As stated in Proposition \ref{prop:sol:sigma0=0}, one can construct three distinct equilibria to the LQ-MFG. A way to choose an equilibrium among the three available ones is to find which one(s) yield the minimal cost. 
This is what we call below \textit{the cost minimization approach}. 

In our specific framework, we have the following result:

\begin{prop}
\label{mincostselect}
With the notations of Proposition 
\ref{prop:sol:sigma0=0}, the cost minimization approach selects the equilibrium corresponding to $A=0$, i.e  $$\big(\mu_t = 0 , \alpha_t =  -\eta_t X_t \big)_{t\in[0,T]} \quad \textrm{\rm where} \quad X_t  =  w_t \int_0^t \sigma w_s^{-1} dW_s, \quad  \forall t \in [0,T].$$
\end{prop}

\begin{proof}
Given, $A \in \{ -1,0,1\},$ the dynamics of the representative player, $(X_t)_{t\in[0,T]}$, and the controls, $ (\alpha_t)_{t\in[0,T]} $, at equilibrium are given by 
\begin{equation*}
\begin{cases}
&X_0 = 0, 
\\
&dX_t  = \big[ (\kappa-\eta_t)X_t - A w_t^{-1}  \big]dt + \sigma dW_t, \quad  \forall t \in [0,T] \notag 
\\
&\alpha_t = -\eta_t X_t - A w_t^{-1} \quad  \forall t \in [0,T].
\end{cases}
\end{equation*}
We recall that the cost functional is given by
$$ J(\alpha)   = \mathbb{E} \Bigg[ \int_0^T \frac{1}{2} [\alpha_t^2 + ( X_t)^2 ] dt  + \frac{1}{2}(X_T + g(\mu_T))^2 \Bigg], \quad \alpha \in \mathcal{H}^2 . $$
By replacing the control at equilibrium in the cost functional, we get (with an obvious notation 
for $J_{A}$)
$$ J_{A} =  \mathbb{E} \Bigg[ \int_0^T \frac{1}{2} [(-\eta_t X_t - w_{t}^{-1} A )^2 + ( X_t)^2 ] dt  + \frac{1}{2}(X_T + A)^2 \Bigg], \quad A \in \{-1,0,1\}. $$
In order to expand $J_{A}$, we recall that 
\begin{equation*}
\begin{split}
&{\mathbb E}[X_{t}] = \mu_{t} =  
-A w_{t }\int_0^t w^{-2}_s ds,
\\
&{\mathbb V}[X_{t}]= {\mathbb E}\bigl[\bigl( X_{t} - \mu_{t}\bigr)^2 \bigr] = w_{t}^2 \sigma^2 \int_{0}^t w_{s}^{-2} ds,
\quad \forall t \in [0,T].
\end{split}
\end{equation*}
We then expand $J_{A}$
 as follows
\begin{align*}
J_{A}   &=  \mathbb{E} \Bigg[ \int_0^T \frac{1}{2} [(-\eta_t X_t - w_{t}^{-1} A )^2 + ( X_t)^2 ] dt  + \frac{1}{2}(X_T + A)^2 \Bigg] \notag \\
& = \frac{1}{2} \mathbb{E} \Bigg[ \int_0^T \bigg( (1+\eta_t^2)X^2_t + 2 A \eta_t w_t^{-1} X_t + A^2 w_t^{-2}  \bigg)  dt  \Bigg] + \frac{1}{2} \mathbb{E} \big[ (X_T + A)^2 \big] \notag \\
& = \frac{1}{2} \int_0^T \bigg( (1+\eta_t^2)\mathbb{E}[(X_t-\mu_t)^2] + (1+\eta_t^2) \mu_t^2  +2 A \eta_t w_t^{-1} \mu_t + A^2 w_t^{-2}  \bigg)  dt  
+ \frac{1}{2} \mathbb{E} \big[ (X_T + A)^2 \big],
 \notag 
\end{align*}
and then \begin{align*}
J_{A} & =  \frac{A^2}{2} \bigg[ \int_0^T  
\biggl( (1+\eta_t^2) w^2_t \bigg( \int_0^t w_s^{-2} ds\bigg)^2  - 2 \eta_t \bigg( \int_0^t w_s^{-2} ds\bigg) + w_t^{-2}\biggr) dt \bigg]  
\\
&\hspace{1cm}
+ \frac{A^2}{2} \bigg( 1 - w_{T}\int_0^T w_s^{-2}ds \bigg)^2   + \frac{1}{2} w_{T}^2 \int_0^T \sigma^2 w_t^{-2}dt + \frac{1}{2} \int_0^T (1+\eta_t^2) \sigma^2 w_t^{2} \bigg( \int_0^t w_s^{-2} ds \bigg) dt,
\end{align*}
where we used the fact that
\begin{align*}
\mathbb{E} \big[ (X_T + A)^2 \big] 
&= \bigl( A + \mu_{T} \bigr)^2
+\mathbb{E} \big[ (X_T - \mu_{T})^2 \big] 
\\
& = A^2 \bigg( 1 - w_{T} \int_0^T w_s^{-2}ds \bigg)^2  + w_{T}^2 \int_0^T \sigma^2 w_t^{-2}dt, \quad \forall t \in [0,T].
\end{align*}
In order to conclude, it remains to take into account the fact that 
\begin{equation*}
\begin{split} 
&(1+\eta_t^2) w^2_t \bigg( \int_0^t w_s^{-2} ds\bigg)^2  - 2 \eta_t \bigg( \int_0^t w_s^{-2}ds \bigg) + w_t^{-2}
\\
&= w^2_t \bigg( \int_0^t w_s^{-2} ds\bigg)^2
+ \biggl( \eta_{t} w_{t} \int_0^t w_s^{-2} ds - w_{t}^{-1} \biggr)^2 > 0, \quad \forall t \in [0,T].
\end{split}
\end{equation*}
One concludes that $J_{A}$, for $A \in \{-1,0,1\},$ is minimal when $A = 0$, and $J_{\pm1} > J_{0}$, which completes the proof. 
\end{proof}

\section{Master equation and related PDE estimates}
\label{se:PDE}

In this section, 
we consider the case $\sigma_{0} \in (0,1)$ in the LQ-MFG under study. By \cite{Foguen}, we know that there is a unique 
equilibrium to the LQ-MFG, which is described by FBSDE (\ref{fbsde2}). 

Since (\ref{fbsde2}) is uniquely solvable, we can use Ma-Protter-Yong's four-step-scheme 
\cite{MaProtterYong}
to represent the solution, see also \cite{MaWuZhangZhang}. The four-step-scheme provides a so-called decoupling field that decouples the two forward and backward equations of the FBSDE, meaning that it permits to represent the backward component of the solution in terms of the forward one. Due to the diffusive effect of the Brownian motion $(B_{t})_{t \in [0,T]}$, such a decoupling field is smooth. Through the Cole-Hopf transformation (which we make clear below), it can be represented explicitly and then inserted into the FBSDE (\ref{fbsde2}): This allows to read the forward component of (\ref{fbsde2}) as a standard a SDE. 

\subsection{Master equation}

The first step is to make the connection between the aforementioned decoupling field and 
the notion of master equation.

The concept of master equation was introduced by Lions \cite{Lions} in his lectures on mean-field games 
at \textit{Coll\`ege de France}. Generally speaking, the master equation is an equation for the 
value of the mean-field game. It is regarded as a function of the initial conditions of the game, which include: Initial time, 
initial state of the representative player and initial state of the population. To ensure that the value function indeed makes sense, equilibria must be unique. 

In our case, $\sigma_{0} \in (0,1)$ and the LQ-MFG has a unique equilibrium. Still, we prefer to write down an equation 
for the (optimal) feedback function of the LQ-MFG instead of an equation for the value function. In fact, both are related with one
another through a standard minimization argument of the Hamiltonian and, in our framework,  
the (optimal) feedback function is given by the opposite of the derivative of the value function, the derivative being taken
with respect to the private state of the representative player. We refer to \cite[Chapters 3 and 4]{CarmonaDelarue_book_I}
for details. 

Actually, we know from Proposition \ref{prop1}
that the equilibrium strategy of the LQ-MFG must be of the form 
\begin{equation*}
\alpha_{t} = - \eta_{t} X_{t} - w_{t}^{-1} h_{t}^{\xi,\sigma_{0}},
\quad \forall t \in [0,T],
\end{equation*}
where $h^{\xi,\sigma_{0}}$ solves the backward equation in 
\eqref{fbsde2} (with a general initial condition $\xi$ instead of $0$ for the forward process). 

Now, as we recalled right above, it is a standard fact from FBSDE theory, see for instance \cite{Delarue02,MaProtterYong}, that 
the backward process can be put in the form
\begin{equation*}
h_{t}^{\xi,\sigma_{0}} = \theta^{\sigma_{0}}(t,\mu_{t}^{\xi,\sigma_{0}}), \quad
\forall t \in [0,T].
\end{equation*}
Here, $\theta^{\sigma_0}$ is called the decoupling field of the FBSDE 
\eqref{fbsde2}. At the end of the day, the function
\begin{equation*}
[0,T] \times \RR \times {\mathcal P}_{1}(\RR) \ni (t,x,m) \mapsto 
- \eta_{t} x - w_{t}^{-1} \theta^{\sigma_{0}}(t,x,\bar m),
\end{equation*}
is the right candidate for solving the master equation (for the feedback function). 
Here ${\mathcal P}_{1}(\RR)$ is the space of probability 
measures on $\RR$ with a finite first moment and $\bar m$ stands for the mean 
of $m$ when $m \in {\mathcal P}_{1}(\RR)$. 
In fact, instead of writing down the full master equation (which is a difficult object, see 
\cite{CardaliaguetDelarueLasryLions,ChassagneuxCrisanDelarue,GangboSwiech}
and \cite[Chapter 5]{CarmonaDelarue_book_II}), we just write down the equation for 
$\theta^{\sigma_{0}}$, which is enough for our own purpose. 
To do so, notice from 
\cite{Delarue02,DelarueGuatteri}, see also 
the book 
\cite{Ladyzhenskaya}, 
that $\theta^{\sigma_{0}}$ belongs to $\mathcal{C}^{1,2}( [0,T) \times \mathbb{R} ; \mathbb{R})
 \cap {\mathcal C}([0,T] \times {\mathbb R};{\mathbb R})$
and is a 
 classical solution to the following quasilinear parabolic PDE with terminal condition :
\begin{align}
\label{PDE}
\begin{cases}
\forall (t,x) \in [0,T) \times \mathbb{R},
\\
\partial_t \theta^{\sigma_0}(t,x) - w^{-2}_t \theta^{\sigma_0}(t,x) \partial_x \theta^{\sigma_0}(t,x) + \frac{1}{2} \sigma_0^2 w^{-2}_t \partial^{2}_{xx}\theta^{\sigma_0}(t,x) = 0, 
\\
\theta^{\sigma_0}(T, x) = g(x), \hspace{2mm} \forall x \in \mathbb{R}. 
\end{cases}
\end{align}
This PDE is well-known in the literature: it is a Burgers type PDE. For theoretical and numerical entry points on the analysis of Burgers type PDEs, we refer to the textbooks \cite{HormanderBook,Lax}, and to the article \cite{BossyTalay}. It is uniquely solvable and a representation of its solution is obtained through the Cole-Hopf transformation. For every $t \in [0,T)$, we write $r_t = \int_t^T w_s^{-2}ds$.  This representation reads as follows:
\begin{equation}
\label{teta0}
\theta^{\sigma_0}(t,x) =  \frac{\int_{\mathbb{R}} (\frac{x-y}{r_t}) \exp( \sigma_0^{-2} (-\int_0^y g(v)dv - \frac{(x-y)^2}{2 r_t}))dy}{\int_{\mathbb{R}}\exp( \sigma_0^{-2} (-\int_0^y g(v)dv - \frac{(x-y)^2}{2 r_t}))dy} .
\end{equation}
Observe from an obvious change of variable that we can easily reduce 
\eqref{PDE} (respectively \eqref{teta0}) to the classical forward viscous Burgers equation (respectively to the classical Cole-Hopf formula). It thus suffices to consider
the function $(t,x) \in [0,r_{0}] \mapsto \theta^{\sigma_{0}}(r_{t}^{\circ -1},x)$, where 
$[0,r_{0}] \ni t \mapsto r_{t}^{\circ -1}$ is the converse of $[0,T] \ni t \mapsto r_t$. This is extremely useful in order to invoke known results from the literature on standard inviscid and viscous Burgers equations. 

\subsection{A priori bounds}
We here collect several key estimates for $\theta^{\sigma_{0}}$. The first one is 

\begin{lem}
\label{lem:bounded}
The function $\vert \theta^{\sigma_{0}} \vert$ is bounded by 1, for any $\sigma_{0} >0$.
\end{lem}

\begin{proof}
The proof follows from the fact that $\vert g\vert$ is bounded by $1$ and that $\theta^{\sigma_{0}}$ is obtained by transporting $g$ along the forward component of 
\eqref{fbsde2}.
\end{proof}

Things are much worse for the first-order derivative (and in fact this is the reason why the analysis of the case 
$\sigma_{0}=0$ is so difficult). In fact, by standard results in 
the 
theory of nonlinear parabolic equations, see
for instance the monograph \cite{Ladyzhenskaya}, see also \cite{Delarue02,DelarueGuatteri,MaProtterYong} for a probabilistic 
point of view, ${\theta}^{\sigma_{0}}$ is Lipschtiz continuous in 
space, uniformly in time, but the \textit{Lipschitz constant depends on $\sigma_{0}$!}
Still, we have the following bound that gives a bound on the rate of explosion as $\sigma_{0}$ tends to $0$. 

\begin{lem}
\label{lem:gradient}
There exists a constant $C$ such that
\begin{equation*}
\vert \partial_{x} \theta^{\sigma_{0}}(t,x) \vert \leq \frac{C}{\sigma_{0}^2}, \quad \forall (t,x) \in [0,T) \times \RR. 
\end{equation*} 
\end{lem}

\begin{proof}
We perform the change of variable:
\begin{equation*}
\hat{\theta}^{\sigma_{0}}(t,x) = {\theta}^{\sigma_{0}}(\sigma_{0}^2 t, \sigma_{0}^2 x),
\quad
\forall (t,x) \in \bigl[0,\frac{T}{\sigma_{0}^2}\bigr) \times \mathbb{R}.
\end{equation*}
Then,
\begin{align*}
\begin{cases}
 \partial_t  \hat\theta^{\sigma_0}(t,x) -   w^{-2}_t \hat{\theta}^{\sigma_0}(t,x) \partial_x \hat{\theta}^{\sigma_0}(t,x) + \frac{1}{2}   w^{-2}_t \partial^{2}_{xx}\hat{\theta}^{\sigma_0}(t,x) & = 0, \hspace{3mm} \forall (t,x) \in \bigl[0,\frac{T}{\sigma_{0}^2}\bigr) \times \mathbb{R}, 
\\
\hat{\theta}^{\sigma_0}\bigl( \frac{T}{\sigma_{0}^2}, x\bigr) = g(\sigma_{0}^2 x), \hspace{2mm} \forall x \in \mathbb{R}. 
\end{cases}
\end{align*}
Now, the result follows from standard PDE estimates for uniformly parabolic equations, see the same references as before:
\cite{Ladyzhenskaya} for PDE arguments and \cite{Delarue02,DelarueGuatteri,MaProtterYong} for the probabilistic point of view.
\end{proof}

Although the gradient may blow up, we have in fact an upper bound for it.

\begin{lem}
\label{lem:non:decreasing}
For any $\sigma_{0} \in (0,1)$,
the function $\theta^{\sigma_{0}}$ is non-increasing in $x$. 
\end{lem}

\begin{proof}
As $\theta^{\sigma_{0}}$ is the decoupling field of 
\eqref{fbsde2}, we have:
\begin{equation*}
{\theta}^{\sigma_{0}}(t,x) = {\mathbb E} \bigl[ g(\mu_T^{t,x,\sigma_{0}}) \bigr],
\end{equation*}
where 
\begin{equation}
\label{eq:t,x,SDE}
d \mu_{s}^{t,x,\sigma_{0}} = - w_s^{-2} \theta^{\sigma_{0}} \bigl(s, \mu_{s}^{t,x,\sigma_{0}} \bigr) ds + 
\sigma_{0} w_s^{-1}d{B}_{s}, \quad s \in [t,T],
\end{equation}
Using standard results for one-dimensional SDEs driven by Lipschitz coefficients, we know that $x \leq y$ implies $\mu_{T}^{t,x,\sigma_{0}} \leq \mu_{T}^{t,y,\sigma_{0}}$ with probability 1. Since  $g$ is non-increasing, we complete the proof. 
\end{proof}

\subsection{Zero-noise limit of the decoupling field}

It is a well-known fact that, as $\sigma_{0}$ tends to $0$, $\theta^{\sigma_{0}}$ converges 
(in a sense that is made clear below) to the so-called \textit{entropy solution}
of the inviscid version of \eqref{PDE}. Again, we refer to \cite{HormanderBook,Lax}. 
The limit is given by the field $\theta$, whose definition is as follows. 
For all $ (t,x) \in ([0,T]\times \mathbb{R})$, we let:

\begin{equation}
\label{teta}
    \theta(t,x)= \begin{cases}
                
                 -\textrm{\rm sign}(x)   \hspace{11mm} &\text{if}  \hspace{5mm} t \leq \delta, \hspace{2mm} x \in \mathbb{R},
                  \\ 
                 
                -\textrm{\rm sign}(x) \hspace{11mm} &\text{if} \hspace{5mm} t \geq \delta,  \hspace{2mm}   
                |x| \geq r_{\delta} - r_t,
                 \\
                  
                -\frac{x}{r_{\delta} - r_t} &\text{if}  \hspace{5mm} t > \delta,  \hspace{2mm}   |x| < r_{\delta} - r_t,
   \end{cases}
\end{equation}
where we recall the definition of $r_{\delta}$ in \eqref{eq:g}.



Most of our analysis for the zero-noise limit of the 
LQ-MFG with common noise as $\sigma_{0}$ tends to $0$ is based upon sharp estimates of the difference between the fields $\theta^{\sigma_0}$ and $ \theta$. 
In this regard, we have the following bound on the difference 
$\theta^{\sigma_0}(t,x) - \theta(t,x)$,  for $\sigma_0 \in (0,1)$ and for some $(t,x)  \in [0,T)\times \mathbb{R}$.

\begin{prop} 
\label{propPSI}
Let $ \Psi (t, x, \sigma_0):=\theta^{\sigma_0}(t,x) - \theta(t,x)$, for  $(t,x,\sigma_0) \in [0,T)\times\mathbb{R}\times(0,1).$ 
Then, 
for any non-negative non-decreasing curve $\psi \in {\mathcal C}([0,T];\RR)$, 
which is strictly above the curve $[0,T] \ni t \mapsto (r_{\delta} - r_{t})_{+}$ on a left-open interval containing $[\delta,T]$, and for any function $L$ from $(0,+\infty)$ into itself such that 
$\lim_{\sigma_{0} \rightarrow 0} L(\sigma_{0})=+\infty$,
\begin{equation*}
\lim_{\sigma_{0} \rightarrow 0}
\sup_{(t,x) \in [0,T], \vert x \vert \geq \sigma_{0}^2 L(\sigma_{0}) + \psi_{t}}
\vert \Psi(t,x,\sigma_{0}) \vert =0.
\end{equation*}
\end{prop}

\begin{proof}
Following \cite{TadmorTang,TadmorTang2}, we know that, for any $\eta >0$, 
\begin{equation}
\label{eq:TadmorTang}
\lim_{\sigma_{0} \rightarrow 0}
\sup_{t \in [0,T], \vert x \vert \geq \eta + \psi_{t}}
\vert \Psi(t,x,\sigma_{0}) \vert =0.
\end{equation}
The proof of \eqref{eq:TadmorTang} is in fact rather straightforward in our setting and we give it for completeness: 
The first point is to observe that, when $t \in (\delta,T]$, 
the system 
\eqref{fbsde2:sigma0=0} initiated at time $t$ from any $x \in {\mathbb R}$ is well-posed
and that the value of the backward process at time $t$ then coincides with $\theta(t,x)$. This
is a well-known fact in the theory of hyperbolic equations, which also follows from the small time analysis performed in 
\cite{Delarue02} for FBSDEs with Lipschitz continuous coefficients. 
So, for $t \in [\delta+\epsilon,T]$ for some $\epsilon \in (0,T-\delta)$, we can make the difference between 
the two systems 
\eqref{fbsde2}
and
\eqref{fbsde2:sigma0=0}
with $(t,x)$ instead of $(0,0)$ as initial condition. Following \cite{Delarue02}, we can prove that 
\begin{equation*}
\lim_{\sigma_{0} \rightarrow 0}
\sup_{(t,x) \in [\delta+\epsilon,T] \times \RR}
\vert \Psi(t,x,\sigma_{0}) \vert =0.
\end{equation*}
Now, we can choose $\epsilon$ small enough such that $r_{\delta} - r_{\delta+\epsilon} < \eta/2$. Then, for $t < \delta+\epsilon$ and $x \geq \eta + \psi_{t}$, 
\begin{equation*}
\begin{split}
\bigl\vert \Psi(t,x,\sigma_{0}) \bigr\vert = \bigl\vert \theta^{\sigma_{0}}(t,x) + 1 \bigr\vert &=
\Bigl\vert {\mathbb E}
\bigl[ \theta^{\sigma_{0}}\bigl(\delta+\epsilon,\mu_{\delta+\epsilon}^{t,x,\sigma_{0}}
\bigr) + 1 \bigr]
\Bigr\vert 
\\
&\leq 
\sup_{(s,y) \in [\delta+\epsilon,T] \times \RR}
\vert \Psi(s,y,\sigma_{0}) \vert +  4 {\mathbb P} \Bigl( \mu_{\delta+ \epsilon}^{t,x,\sigma_{0}} \leq \frac{\eta}{2} \Bigr),
\end{split}
\end{equation*}
where we used the same notation as in \eqref{eq:t,x,SDE}.
Using the fact that $\theta^{\sigma_{0}}$ is non-positive in $[0,T] \times \RR_{+}$, we have
\begin{equation*}
{\mathbb P} \Bigl( \mu_{\delta+ \epsilon}^{t,x,\sigma_{0}} \leq \frac{\eta}{2} \Bigr)
\leq 
{\mathbb P} \Bigl( \inf_{s \in [t,\delta+\epsilon]} \mu_{s}^{t,x,\sigma_{0}} \leq \frac{\eta}{2} \Bigr)
\leq 
{\mathbb P} \biggl(\sigma_{0}  \sup_{s \in [t,\delta+\epsilon]} \biggl\vert \int_{t}^s w_{r}^{-1} dB_{r}
\biggr\vert  \geq
 \frac{\eta}{2} \biggr),
 \end{equation*}
which suffices to get 
\eqref{eq:TadmorTang}.

As a consequence, it suffices to prove that for any $\eta \in (0,\delta)$, 
\begin{equation}
\label{eq:propPSI2}
\lim_{\sigma_{0} \rightarrow 0}
\sup_{t \in [0,\eta], \vert x \vert \geq \sigma_{0}^2 L(\sigma_{0}) }
\vert \Psi(t,x,\sigma_{0}) \vert =0,
\end{equation}
which is done in 
appendix, see 
Section
\ref{se:appendix}.
\end{proof} 
%

\begin{remark}
Obviously, the proof of Proposition \ref{propPSI}
provides a stronger result than what the statement claims, but the statement 
will suffice for our purpose. In fact, we feel better to state in a minimal way the conditions that 
we need to establish Theorem \ref{thm:4}. Moreover, we stress the fact that the choice of 
the terminal condition here plays a crucial role in the proof of 
\eqref{eq:propPSI2}, see again the appendix. In this regard, it is worth mentioning that 
there are numerous references on the convergence of viscous solutions to entropy solutions of scalar conservation laws, see for instance 
\cite{Kuznetsov,TadmorTang,TadmorTang2,TangTeng}: In comparison, 
\eqref{eq:propPSI2} is a fine estimate and a careful inspection would be needed to determine to the precise class of functions $g$ for which our methodology could be applied. 
\end{remark}

\subsection{$L^1$ stability}

In the analysis, we shall make use of the following lemma, which is a key property of scalar conservation laws, 
see for instance \cite{HormanderBook,Lax}. 

\begin{lem}
\label{comparePDEs}
Consider a Lipschitz continuous bounded function $ \tilde{g} \in \mathcal{C}(\mathbb{R}; \mathbb{R})$ such that $\tilde{g} \geq g $, and, for $\sigma_{0} \in (0,1)$, call $\tilde{\theta}^{\sigma_0}(t,x)$ the classical solution to 
\begin{align}
\label{PDE2}
\begin{cases}
\forall t \in [0,T) \times \mathbb{R},
\\
\partial_t \tilde{\theta}^{\sigma_0}(t,x) - w^{-2}_t \tilde{\theta}^{\sigma_0}(t,x) \partial_x \tilde{\theta}^{\sigma_0}(t,x) + \frac{1}{2} \sigma_0^2 w^{-2}_t \partial^{2}_{xx}\tilde{\theta}^{\sigma_0}(t,x)  = 0,
\\
\tilde{\theta}^{\sigma_0}(T, x) = \tilde{g}(x), \hspace{2mm} \forall x \in \mathbb{R}. 
\end{cases}
\end{align}
Then, the difference $\tilde{\theta}^{\sigma_0}-  \theta^{\sigma_0}$ is preserved, that is
\begin{equation*}
\bigl( \tilde{\theta}^{\sigma_0}-  \theta^{\sigma_0} \bigr)(t,x) \geq 0, \quad \forall(t,x) \in [0,T] \times \RR.  
\end{equation*}
Moreover,
if $\tilde{g}$ and $g$ coincide outside a compact subset of $\RR$, then 
 the  space integral of  $\vert \tilde{\theta}^{\sigma_0}-  \theta^{\sigma_0}\vert$ is also preserved, i.e
$$\int_{-\infty}^{+\infty} \bigl\vert \bigl(\tilde{\theta}^{\sigma_0} - \theta^{\sigma_0}\bigr)(t,x) \bigr\vert dx 
\leq  \int_{-\infty}^{+\infty} \bigl\vert (\tilde{g} - g)(x) \bigr\vert dx, \quad \textrm{\rm for all} \quad \forall t \in [0,T].$$
\end{lem}

In fact, the last inequality is an equality (because $\tilde{\theta}^{\sigma_0}-  \theta^{\sigma_0}$
has a constant sign), but we won't use this fact in the sequel.

\begin{proof}
\textit{First step.}
The fact that \eqref{PDE2} is well-posed is a standard fact in the theory of nonlinear parabolic equations, see
for instance \cite{Ladyzhenskaya}, see also \cite{Delarue02,DelarueGuatteri,MaProtterYong} for the probabilistic interpretation.
Importantly, since $\tilde{g}$ is Lipschitz continuous, $\tilde{\theta}^{\sigma_{0}}$ is also Lipschtiz continuous in 
space, uniformly in time. 

We then observe that the difference $\tilde{\theta}^{\sigma_{0}} - \theta^{\sigma_{0}}$ is the solution of 
\begin{equation*}
\partial_t \bigl( \tilde{\theta}^{\sigma_0} - \theta^{\sigma_{0}} \bigr) - w^{-2}_t \tilde{\theta}^{\sigma_0} \partial_x \bigl( \tilde{\theta}^{\sigma_0} - \theta^{\sigma_{0}} \bigr) + \frac{1}{2} \sigma_0^2 w^{-2}_t \partial^{2}_{xx}
\bigl( \tilde{\theta}^{\sigma_0} -{\theta}^{\sigma_0} 
\bigr) -  w_{t}^{-2}  \partial_{x} \theta^{\sigma_{0}} \bigl(\tilde \theta^{\sigma_{0}} - \theta^{\sigma_{0}} \bigr)= 0,
\end{equation*}
which can be regarded as a linear equation in $\tilde{\theta}^{\sigma_{0}} - \theta^{\sigma_{0}}$. 
Since 
$\tilde{g} \geq g$, we deduce from the maximum principle that $\tilde{\theta}^{\sigma_{0}} \geq \theta^{\sigma_{0}}$. 
\vspace{5pt}

\textit{Second step.}
The second part of the proof is a direct consequence of Theorem 3.3.1
in 
\cite{HormanderBook}.
\end{proof}

\section{Zero noise limit selection}
\label{se:zero:noise}

The zero noise limit problem described in Section \ref{se:2} requires to find the limit of the unique equilibrium (when $ \sigma_0 >0$) as $\sigma_0 \rightarrow 0 $ (i.e as the common noise $B$ influence on the players vanishes). Through Proposition \ref{prop1}, it is equivalent to study the limit of the unique solution of FBSDE (\ref{fbsde2}) as $\sigma_0 \rightarrow 0 $. 
In this regard, the previous section allows us to reduce the problem to the analysis of the zero-noise limit of the forward SDE 
\begin{equation}
\label{eq:SDE:FSS}
dX_{t} = - w_{t}^{-2} \theta^{\sigma_{0}}(t,X_{t}) dt + \sigma_{0} w_{t}^{-1} dB_{t},
\end{equation}
which we derived from the four-step-scheme.  This prompts us to use the asymptotic form of the decoupling field as $\sigma_0$ tends to $0$, as 
discussed in Proposition
\ref{propPSI}, in order to study the asymptotic behaviour of \eqref{eq:SDE:FSS}. The difficulty to do so comes from the fact that the limit of the decoupling field is discontinuous at $x=0$ and
for $t$ close to $0$. In particular, similar to the famous Peano example for differential equations (the situation is even worse since 
the limiting drift is in fact non-continuous whilst Peano example is for an ODE with a continuous drift), the zero-noise limit of the forward SDE is not uniquely solvable. 
Taking benefit of the fact that this SDE is set in dimension 1, we manage to adapt the techniques from \cite{DelarueFlandoli} to determine the solutions of the asymptotic forward SDE that are selected in the limit. Precisely, we show that this approach selects the extremal equilibria (i.e $ A \in \{1, -1 \}$) in 
Proposition \ref{prop:sol:sigma0=0}.
Using the same 
 terminology as in \cite{DelarueFlandoli},
  we also exhibit a transition space-time point,  the precise definition of which is given in the next subsection.

\subsection{Main result}

Thanks to the four-step-scheme, the solution to FBSDE (\ref{fbsde2}) can now be identified 
with the solution of \eqref{eq:SDE:FSS} with $0$ as initial condition at time $0$. More 
precisely, thanks to Proposition 
\ref{propPSI}, 
we can write
$(\mu^{\sigma_0}_t:=\mu^{0,\sigma_0}_t)_{t \in [0,T]}$
as the solution of the SDE:  
\begin{align}
\label{SDE}
\begin{cases}
\mu^{\sigma_0}_t &= - \int_0^t  w^{-2}_s  \theta^{\sigma_0}(s,\mu^{\sigma_0}_s)  ds + \int_0^t w^{-1}_s \sigma_0 dB_s \hspace{4mm} \\ 
&=  - \int_0^t w^{-2}_s \big( \theta(s,\mu^{\sigma_0}_s)   + \Psi (s, \mu^{\sigma_0}_s, \sigma_0)  \big) ds + \int_0^t w^{-1}_s \sigma_0 dB_s, \hspace{4mm} \forall t \in [0,T].
\end{cases}
\end{align}



As explained before, our objective is to find a transition point for the process $(\mu^{\sigma_0}_t)_{t \in [0,T]}$. A transition point in this setting is a pair of two space-time points $(\pm \epsilon_0,t_0)$, which get closer and closer to $(0,0)$ as $\sigma_0$ tends to $0$ with the following 
two properties:
\begin{enumerate}
\item The probability that 
the process $( \mu^{\sigma_0}_t )_{t \in [0,T]}$ reaches $ \pm \epsilon_0 $ in a time of the same scale as $t_{0}$  tends to $1$ as $\sigma_0 \rightarrow 0$. 
\item Given the fact that the process
$(\mu^{\sigma_0}_t)_{t \in [0,T]}$ hits $\pm \epsilon_{0}$ in a time of order $t_{0}$, 
the probability that it escapes away from $0$
after the hitting time  
tends to $1$ as $\sigma_0 \rightarrow 0$.

For the latter item, we will compare the trajectories of $(\mu^{\sigma_0}_t)_{t \in [0,T]}$ with the curves $(\pm \gamma \int_{0}^t w_{s}^{-2} ds)_{0 \leq t \leq T}$, for any arbitrary $\gamma \in (0,1)$. We will show that, with probability asymptotically equal to 1,  
$(\mu^{\sigma_0}_t)_{t \in [0,T]}$ escapes from $0$ at a faster rate than any of these curves.  
\end{enumerate}

Of course, as $\sigma_0 \rightarrow 0$, the effect of the common noise vanishes and outside a null event, the trajectories of $(\mu^{\sigma_0}_t)_{t \in (0,T]}$ will concentrate on the equilibria $A \in \{-1, 1\}$
in Proposition \ref{prop:sol:sigma0=0}. For symmetry reasons, the two equilibria will be charged with the same probability. 
To make it clear, here is the main result of this section:
 
\begin{thm}
\label{mainzeronoise} Consider $\big( k_t := \int^t_0 w_s^{-2}ds \big)_{t \in [0,T]} \in \mathcal{C}([0,T];\mathbb{R})$.
Then, the sequence of laws  $({\mathbb P} \circ 
(\mu^{\sigma_0}_{t})_{0 \le t \le T}^{-1})_{\sigma_0 \in (0, 1)}$ converges, as $\sigma_0 \rightarrow 0$, to 
\begin{equation*}
\frac12 \delta_{(k_t)_{t \in [0,T]}}
+\frac12 \delta_{(-k_t)_{t \in [0,T]}}.
\end{equation*}


\end{thm}

\subsection{A more general framework}
\label{subse:5:2}
We shall prove Theorem 
\ref{mainzeronoise} as a particular case of a more general framework. 
Assume indeed that, for any $\sigma_{0} \in (0,1)$, there exists a 
continuous
mapping $\tilde{\theta}^{\sigma_{0}}$ from $[0,T] \times \RR$ into $[-1,1]$ 
with the following three features:
\begin{enumerate}
\item[(A1)] 
There exists a function $L$ from $(0,+\infty)$ into itself satisfying 
$$\lim_{\sigma_{0} \rightarrow 0} L(\sigma_{0})=+\infty
\qquad \textrm{and}
\qquad
\lim_{\sigma_{0} \rightarrow 0} L(\sigma_{0}) \vert \ln(\sigma_{0}) \vert^{-1/8}=0,$$
such that, 
for any
non-negative non-decreasing curve $\psi \in {\mathcal C}([0,T];\RR)$, 
which is strictly above the curve  $t \mapsto  (r_{\delta}-r_{t})_{+}$ on a left-open interval containing $[\delta,T]$,
\begin{equation*}
\lim_{\sigma_{0} \rightarrow 0}
\sup_{(t,x) \in [0,T], \vert x \vert \geq \sigma_{0}^2 L(\sigma_{0}) + \psi_{t}}
\vert (\tilde \theta^{\sigma_{0}} - \theta)(t,x) \vert =0.
\end{equation*}
\item[(A2)] For any $\sigma_{0}>0$, there exist three real-valued continuous 
adapted processes 
$(\tilde \mu_{t}^{\sigma_{0}})_{t \in [0,T]}$,
$(\tilde Y_{t}^{\sigma_{0}})_{t \in [0,T]}$,
and $(\triangle_{t}^{\sigma_{0}})_{t \in [0,T]}$, the process 
$(\tilde Y_{t}^{\sigma_{0}})_{t \in [0,T]}$
taking values in $[-1,1]$, 
 such that 
\begin{equation*}
\tilde \mu^{\sigma_{0}}_t =  - \int_0^t w_s^{-2}  \tilde{Y}^{\sigma_{0}}_{s} ds  + 
\sigma_{0} \int_{0}^t w_{s}^{-1} dB_{s}, \quad t \in [0,T], 
\end{equation*}
with 
\begin{equation*}
\tilde Y_{t}^{\sigma_{0}} = \tilde{\theta}^{\sigma_{0}}\bigl(t,\tilde \mu^{\sigma_{0}}_{t}\bigr) - \triangle_{t}^{\sigma_{0}}, \quad t \in [0,T], 
\end{equation*}
and
$\sup_{t \in [0,T]} 
\vert \triangle_{t}^{\sigma_{0}} \vert$ tends to $0$ in probability as $\sigma_{0}$ tends to $0$.
\item[(A3)] For any $\sigma_{0}>0$, the law of 
$\tilde \mu^{\sigma_{0}}$ on ${\mathcal C}([0,T];{\mathbb R})$
is the same as the law of 
$- \tilde \mu^{\sigma_{0}}$.
\end{enumerate}

In this framework, we prove below the following statement:
\begin{thm}
\label{thm:general}
Under assumptions \textrm{\rm (A1)}, \textrm{\rm (A2)} and \textrm{\rm (A3)}, the sequence $({\mathbb P} \circ 
(\tilde \mu^{\sigma_{0}}_{t})_{t \in [0,T]}^{-1})_{\sigma_{0} \in (0,1)}$ converges, 
as $\sigma_{0}$ tends to $0$, to 
\begin{equation*}
\frac12 \delta_{(k_t)_{t \in [0,T]}}
+\frac12 \delta_{(-k_t)_{t \in [0,T]}}.
\end{equation*}
\end{thm}

It is well checked that 
Theorem 
\ref{mainzeronoise}
follows from 
Theorem \ref{thm:general}:
(A1) is Proposition 
\ref{propPSI}, 
(A2) is Lemma 
\ref{lem:bounded}
and 
(A3) is a consequence of uniqueness (in law) 
to 
\eqref{fbsde2}, 
noticing that the triple 
$(-\mu^{0,\sigma_0}_t, -h^{0,\sigma_0}_t, Z^{0,\sigma_0}_t  )_{t\in[0,T]}$
is a solution of 
the system
\eqref{fbsde2}
driven by 
$(-B_{t})_{t \in [0,T]}$.

\subsection{Reaching the transition point}
\label{subse:transition:point}
Throughout the subsection, 
we assume that (A1), (A2) and (A3) hold true. 
We define our transition point as
\begin{equation}
\label{eq:transition:point}
(\epsilon_0, t_0) := \bigl(\sigma_0^2 L^2(\sigma_{0}), \sigma_{0} \bigr),
\end{equation} 
where $L$ is as in (A1).

We will regard $\epsilon_{0}$ and $t_{0}$ as functions of $\sigma_{0}$.
We call them infinitesimal functions of $\sigma_{0}$ in the sense that they tend to $0$ with $\sigma_{0}$. 

With the transition point, we associate the following hitting time:
For all $ \psi \in \mathcal{C}([0,T];\mathbb{R})$,
consider 
\begin{equation}
\label{eq:tau1}
\tau_{\epsilon_0}(\psi) := \inf\{ t \in [0, T] : |\psi_t| > \epsilon_0 \},
\end{equation}
with the convention that $\tau_{\epsilon_0}(\psi) := T$
when the set in the right-hand side is empty.

\begin{prop}{\textbf{\bf (Transition point).}}
\label{transitionpoint}
Consider $\tilde{t_0}$, a positive infinitesimal function of $\sigma_0$, such that $\underset{\sigma_0 \rightarrow 0}{\lim} 
\tilde{t_0}/t_0 = +\infty  $.
Then,
$$\mathbb{P} \bigl( \tau_{\epsilon_0}\bigl(\tilde \mu^{\sigma_0}\bigr) > \tilde{t_0} \bigr) \rightarrow 0 \hspace{5mm} \textrm{\rm as} \hspace{5mm} \sigma_0 \rightarrow 0. $$
\end{prop}

\subsubsection{A technical lemma}
The proof of Proposition \ref{transitionpoint} is based upon the following general lemma:

\begin{lem}
\label{le:5}
For a positive continuous (deterministic) path $(\check{w}_{t})_{t \in [0,T]}$, let 
 $(\check{X}_{t})_{t \in [0,T]}$ be a one-dimensional It\^o process of the form
\begin{equation*}
d \check{X}_{t} = - \check w^{-2}_t  \check{Y}_{t} dt + \check w_t^{-1} d \check B_{t}, \quad t \in [0,T],
\end{equation*}
where 
$(\check{B}_{t})_{t \in [0,T]}$ is a Brownian motion 
with respect to some filtration 
$(\check{\mathcal F}_{t})_{t \in [0,T]}$
and
$(\check{Y}_t)_{t\in[0,T]}$ is a $[-1,1]$-valued adapted process. 
For a real $a \geq 1$ and some stopping time $\varrho$ with respect to the filtration 
$(\hat{\mathcal F}_{t})_{t \in [0,T]}$, 
let 
$$\tau := \inf\{ t \geq \varrho : \vert \check{X}_{t} \vert \geq a\}.$$
Then, on the event $\{ \vert \check{X}_{\varrho} \vert < a \}$, we have 
\begin{equation*}
{\mathbb P} \bigl( \tau \leq T \, \vert \, 
\check{\mathcal F}_{\varrho}
 \bigr) \geq \frac1c \exp \bigl( - c (\check k_{T} - \check k_{\varrho}) \bigr) \exp\bigg(-  \frac{c a^2}{ (\check{k}_T - \check{k}_{\varrho})} \bigg), 
\end{equation*}
where $c$ is a strictly positive universal  constant and 
 \begin{equation*}
 \check{k}_{t} = \int_{0}^t \check w_{s}^{-2} ds, \quad t \in [0,T].  
 \end{equation*}
\end{lem}
\begin{proof}
Without any loss of generality, we can assume that $\varrho=0$. 
We then let ${\mathbb Q}$ be the probability measure defined by
\begin{equation*}
\frac{d {\mathbb Q}}{d {\mathbb P}}
= \exp \biggl( \int_{0}^T \check{Y}_{s} \check w^{-1}_s d\check{B}_{s} - \frac12 \int_{0}^T
\vert \check{Y}_{s} \vert^2 \check w^{-2}_s ds \biggr). 
\end{equation*}
Under ${\mathbb Q}$, the process
\begin{equation*}
\check{W}_{t} = \check{B}_{t} - \int_{0}^t \check{w}_{s}^{-1} \check{Y}_{s} ds ,
\quad t \in [0,T],
\end{equation*}
is a Brownian motion with respect to the filtration $(\check{\mathcal F}_{t})_{t \in [0,T]}$. Moreover, 
\begin{equation*}
\check{X}_{t} - \check{X}_{0} = 
\int_{0}^t   \check w_{s}^{-1} d \check{W}_{s}, 
\quad t \in [0,T].
\end{equation*}
We then obtain 
\begin{equation*}
\begin{split}
{\mathbb Q} \Bigl( \sup_{t \in [0,T]}  \bigl( \check{X}_{t} - \check{X}_{0} \bigr) \geq a \, \big\vert \, \check{\mathcal F}_{0} \Bigr) 
&=  {\mathbb Q} \biggl( 
\sup_{t \in [0,T]}
\biggl[ \int_{0}^t \check{w}_{s}^{-1} d \check{W}_{s} 
\biggr]
 \geq a
\, \big\vert \, \check{\mathcal F}_{0}
 \biggr)
 \\
 &= {\mathbb Q} \biggl( 
\sup_{t \in [0,T]}
\biggl[ \int_{0}^t   \check{w}_{s}^{-1} d \check{W}_{s} 
\biggr]
 \geq a
 \biggr).
\end{split}
\end{equation*}
Then, by Gaussian estimates, we obtain
\begin{equation*}
\begin{split}
{\mathbb Q} \Bigl( \sup_{t \in [0,T]}  \bigl( \check{X}_{t} - 
\check{X}_{0}  \bigr) \geq a \, \big\vert \, \check{\mathcal F}_{0}  \Bigr) 
&\geq {\mathbb Q} \biggl( 
\int_{0}^T   \check{w}_{s}^{-1} d \check{W}_{s} 
 \geq a \biggr)
 \\
 &\geq \frac{1}{\sqrt{2 \pi \check k_{T}}}
 \int_{a}^{a+1} \exp \bigl( - \frac{x^2}{2 \check{k}_{T}}
 \bigr) dx
 \geq \frac{1}{\sqrt{2 \pi \check k_{T}}}
  \exp \bigl( - \frac{2 a^2}{ \check{k}_{T}}
 \bigr).
\end{split}
\end{equation*}
Now,
\begin{equation*}
\begin{split}
{\mathbb Q} \Bigl( \sup_{t \in [0,T]} \bigl( \check{X}_{t} -  \check{X}_{0} \bigr) \geq a 
\, \big\vert \, \check{\mathcal F}_{0}
\Bigr)
&= {\mathbb E}^{\mathbb P}
\Bigl[ \frac{d {\mathbb Q}}{d {\mathbb P}} {\mathbf 1}_{\{ 
 \sup_{t \in [0,T]} ( \check{X}_{t} - \check{X}_{0} ) \geq a
\}} \, \big\vert \, \check{\mathcal F}_{0} \Bigr]
\\
& \leq {\mathbb E}^{\mathbb P}
\Bigl[ \Bigl( \frac{d {\mathbb Q}}{d {\mathbb P}}\Bigr)^2 \, \vert \, \check{\mathcal F}_{0} \Bigr]^{1/2}
{\mathbb P} \Bigl( \sup_{t \in [0,T]} ( \check{X}_{t} - \check{X}_{0} ) \geq a \, \big\vert \, \check{\mathcal F}_{0} \Bigr)^{1/2}. 
\end{split}
\end{equation*}
It is completely standard to prove that
\begin{equation*}
{\mathbb E}^{\mathbb P} \Bigl[ \Bigl( 
\frac{d {\mathbb Q}}{d {\mathbb P}}
\Bigr)^2 \, \big\vert \, \check{\mathcal F}_{0} \Bigr] \leq \exp (2\check{k}_T). 
\end{equation*}
So, we end up with 
\begin{equation*}
\begin{split}
{\mathbb P} \Bigl( \sup_{t \in [0, T]} \bigl( \check{X}_{t} - \check{X}_{0} \bigr) \geq a
\, \vert \, \check{\mathcal F}_{0} \Bigr)^{1/2}
&\geq \frac{1}{\sqrt{2 \pi \check k_{T}}}
 \exp(-\check{k}_T)\exp\bigl(- \frac{2a^2}{  \check{k}_T} \bigr), 
\end{split}
\end{equation*}
which completes the proof on the event $\check X_{0} \geq 0$. Changing $(\check{X}_{t})_{t \in [0,T]}$ into $(-\check{X}_{t})_{t \in [0,T]}$, we easily tackle the case when $\check{X}_{0} \leq 0$.  
\end{proof}

\subsubsection{Proof of Proposition \ref{transitionpoint}}
\begin{proof}
Recall by assumption (A2) that
\begin{align*}
d\tilde \mu^{\sigma_{0}}_t & = - w_t^{-2} \tilde Y^{\sigma_{0}}_t dt + \sigma_{0} w_t^{-1}
 d B_t, \quad \forall t \in [0,T]. \notag 
\end{align*}
By using the change of variables, $$\check X_t := \sigma_{0}^{-2} \tilde \mu^{{\sigma_{0}}}_{\sigma_{0}^2 t}, \quad 
\check Y_t  := \tilde Y^{\sigma_{0}}_{\sigma_{0}^2 t}, \quad \check w_t : = 
w_{\sigma_{0}^2 t}, \quad \check{B}_t := \sigma_{0}^{-1} {B}_{\sigma_{0}^2 t }, \quad \forall t \in [0,\sigma_{0}^{-2} T], $$
it is well checked that 
\begin{equation}
 d \check X_t  = -\check w_t^{-2} \check Y_t  dt + \check w_t^{-1} d\check{B}_t, \quad  \check X_0 = 0, \quad t \in [0, \sigma_{0}^{-2} T].
\end{equation}
Obviously $(\check{B}_{t})_{t \in [0,\sigma_0^{-2} T]}$ is a Brownian motion and 
$(\check X_{t})_{t \in [0,\sigma_0^{-2} T]}$ is an It\^o process for a common rescaled filtration 
$(\check{\mathcal F}_{t})_{t \in [0,\sigma_0^{-2} T]}$.  
We also let
\begin{equation*}
\check{\tau}(\sigma_0) := \inf\bigl\{ t \geq 0 : \vert \check X_{t} \vert \geq   L^2(\sigma_{0})\bigr\} = \sigma_0^{-2} \tau_{\epsilon_{0}}\bigl(\tilde \mu^{\sigma_{0}}\bigr),
\qquad \inf \emptyset = \sigma_{0}^{-2} T.
\end{equation*}

We claim that for the same constant $c>0$ as in the statement of Lemma \ref{le:5}, it holds, for any integer $i \in \{0,\cdots,N-1\}$, with $N:=\lfloor \sigma_{0}^{-2} \rfloor$, 
\begin{equation}
\label{eq:appli:le:5}
\begin{split}
&{\mathbb P} \bigg( \check{\tau}(\sigma_0) \geq \frac{(i+1)}{N \sigma_{0}^2} T \, \bigg| \, \check{\tau}(\sigma_0) > \frac{i T}{N \sigma_{0}^2} \bigg)
\\
&\hspace{15pt} \leq 1- \frac1c \exp \Bigl( - c \sigma_0^{-2}
\bigl(k_{\frac{(i+1)T}{N}} - k_{\frac{i T}{N}} \bigr)
\Bigr)  \exp \bigg(- \frac{c \sigma_{0}^2 L^4(\sigma_{0}) }{
k_{\frac{(i+1)T}{  N}} - k_{\frac{i T}{N}}} 
\bigg).
\end{split}
\end{equation}
The proof works as follows.
We consider the process $(\check{X}_{t})_{t \in [\check \tau(\sigma_{0})\wedge \frac{i T}{N \sigma_{0}^2},\frac{T}{\sigma_{0}^2}]}$
 and we apply Lemma 
\ref{le:5}
with $\varrho = \check \tau(\sigma_{0}) \wedge \frac{i T}{N \sigma_{0}^2}$ and 
$a=L^2(\sigma_{0})$. On the event $\{ \check \tau(\sigma_{0})  >  \frac{i T}{N \sigma_{0}^2} \}$, we get
\begin{equation*}
{\mathbb P}
\biggl(  \check{\tau}(\sigma_0) \leq \frac{(i+1)}{N \sigma^2_{0}} T \, \bigg| \, 
\check{\mathcal F}_{\check{\tau}(\sigma_0) \wedge \frac{i T}{N \sigma_{0}^2}}
 \biggr) \geq 
 \frac1c \exp \bigl( - c
(\check k_{\frac{(i+1)T}{N \sigma_{0}^2}} - \check k_{\frac{i T}{N \sigma_{0}^2}})
\bigr)
  \exp 
 \Bigl( - \frac{c   L^4(\sigma_{0})}{ (
 \check k_{\frac{(i+1)T}{N \sigma_{0}^2}}
 -
 \check k_{\frac{i T}{N \sigma_{0}^2}}
)  } \Bigr),
\end{equation*}
with 
\begin{equation*}
\check{k}_{t} = \int_{0}^{t}
\check{w}_{s}^{-2}
ds =
\int_{0}^t
 w^{-2}_{\sigma_{0}^2 s}
 ds
 = 
\sigma_0^{-2} k_{\sigma_{0}^2 t},
\end{equation*}
from which we get \eqref{eq:appli:le:5}. 
Using the fact that, for $\sigma_{0} \in (0,1)$, $1/N$ is between $\sigma_{0}^2$ and 
$2 \sigma_{0}^2$, we can change the value of the constant $c$ (allowing $c$ to depend 
on $T$ and $\kappa$) so that 
\begin{equation}
\label{eq:appli:le:5:b}
\begin{split}
&{\mathbb P} \bigg( \check{\tau}(\sigma_0) > \frac{(i+1)}{N \sigma_{0}^2} T \, \bigg| \, \check{\tau}(\sigma_0) > \frac{i T}{N \sigma_{0}^2} \bigg)
 \leq 1- \frac1c  \exp \Bigl(- c L^4(\sigma_{0})  
\Bigr).
\end{split}
\end{equation}
By iterating \eqref{eq:appli:le:5:b}, we deduce that 
\begin{equation*}
{\mathbb P} \Bigl( \check{\tau}(\sigma_0) > \frac{i}{N \sigma_{0}^2}T \Bigr) \leq 
\Bigl( 1 -
\frac{1}{c}
\exp \bigl(-  c   L^4(\sigma_{0})
\bigr) \Bigr)^i.
\end{equation*}
Recall that $\ln(1-u) \leq - u$ for $u \in [0,1)$. Thus,
\begin{equation}
\label{eq:27:12:1}
\begin{split}
{\mathbb P} \Bigl( \check{\tau}(\sigma_0) > \frac{i}{N \sigma_{0}^2}T \Bigr) &\leq 
\exp \Bigl( i \ln 
\Bigl( 1 -
\frac1{c}
\exp \bigl(-  c { L^4(\sigma_{0})}
\bigr)\Bigr) \Bigr)\\
& \leq \exp \Bigl( -  \frac{i}{c}
\exp \bigl(-  c { L^4(\sigma_{0})}
\bigr)  \Bigr).
\end{split}
\end{equation}
Choose $i=\lfloor \sigma_{0}^{-1} /(2T) \rfloor$ in \eqref{eq:27:12:1} and deduce that, 
for 
$\sigma_{0}$ small enough,
\begin{equation*}
{\mathbb P} \Bigl( 
 \tau_{\epsilon_{0}}\bigl(\tilde \mu^{\sigma_0}\Bigr) 
\geq \sigma_{0}  
  \bigr) = {\mathbb P} \Bigl(
  \check{\tau}(\sigma_0) \geq 
   \sigma_0^{-1}  
    \Bigr) \leq 
\exp \Bigl[ -  c'{\sigma_{0}^{-1}}
\exp \bigl(-   c  L^4(\sigma_{0})
\bigr)  \Bigr],  
\end{equation*}
for a new constant $c'$. 
Since 
$\lim_{\sigma_{0} \rightarrow 0} L^4(\sigma_{0}) \vert \ln(\sigma_{0}) \vert^{-1/2}=0$, 
the right-hand side tends to $0$ with $\sigma_{0}$, which 
completes the proof.

\end{proof}

\subsection{Restarting from the transition point}
\label{subse:restart}

As before, we assume that (A1), (A2) and (A3) are in force. 
In order to investigate what happens after 
$\tau_{\epsilon_0}(\tilde \mu^{\sigma_0})$, we
 prove first the following lemma:
\begin{lem}
\label{increasingbounds}
There exists a positive constant $ c_\delta \in (0,1) $,
 such that, for all $\gamma \in (0,c_{\delta})$, 
\begin{align*}
(1-\gamma)\int_{\delta/2}^t w^{-2}_s ds \geq \bigl( r_{\delta} - r_{t} 
\bigr)_{+}
+ \frac12 \bigl( r_{\delta/2} - r_{t \wedge \delta} \bigr), \quad t\in \bigl[\frac{\delta}2,T\bigr]. 
\end{align*}

\end{lem}

\begin{proof}
Define $c_\delta$ as  $$0 < c_\delta := \frac{\int_{\delta/2}^\delta w^{-2}_r dr}{2 \int_0^T w^{-2}_r dr} < \frac12. $$
Observe that, whenever  $\gamma \in (0, c_\delta)$,
\begin{equation*}
\gamma \int_{\delta/2}^T w^{-2}_r dr \leq \frac12 \int_{\delta/2}^\delta w^{-2}_r dr,
\end{equation*}
and then, for $t \in [\delta,T]$, 
\begin{equation*}
(1-\gamma)\int_{\delta/2}^t w^{-2}_s ds \geq 
\int_{\delta/2}^t w^{-2}_s ds
- \frac12
\int_{\delta/2}^\delta w^{-2}_r dr \geq 
\int_{\delta}^t w^{-2}_r dr + \frac12 \bigl( r_{\delta/2} - r_{\delta} \bigr).
\end{equation*}
This completes the proof when $t \in [\delta,T]$. Using the fact that $c_{\delta} \leq 1/2$, the result 
is obviously true when $t \in [\delta/2,T]$. 
\end{proof}

The above lemma prompts us to 
introduce (for the same function $L$ as in 
\eqref{eq:transition:point})
\begin{equation}
\label{eq:tau2}
\tau^{+}_{\gamma}(\sigma_{0}) := \inf\Bigl\{ t \in \bigl[
 \tau_{\epsilon_0}\bigl(\tilde \mu^{\sigma_0}\bigr)
, T\bigr] : \tilde \mu^{\sigma_0}_t <   \sigma_0^2 L(\sigma_{0})   + (1- \gamma)\int_{
 \tau_{\epsilon_0}(\tilde \mu^{\sigma_0})
}^t w^{-2}_s ds \Bigr\},
\end{equation}
with the convention that $\tau^{+}_{\gamma}(\sigma_{0}) := T$ if the set is empty.
Similarly, we let
\begin{equation}
\label{eq:tau3}
\tau^{-}_{\gamma}(\sigma_{0}) = \inf\Bigl\{ t \in \bigl[ \tau_{\epsilon_0}\bigl(\tilde \mu^{\sigma_0}\bigr), T\bigr] : \tilde \mu^{\sigma_0}_t > -   \sigma_0^{2} L(\sigma_{0})   - (1- \gamma)\int_{
 \tau_{\epsilon_0}(\tilde \mu^{\sigma_0})
}^t w^{-2}_s ds \Bigr\}.
\end{equation}
\begin{prop}
\label{boundasymp}
For any $\gamma \in (0,c_{\delta})$, it holds that
$$  \mathbb{P}\Bigl( \tau^{+}_{\gamma}(\sigma_0) < T,\tilde \mu^{\sigma_0}_{\tau_{\epsilon_0}(\tilde \mu^{\sigma_0})}
 = \epsilon_{0} \Bigr) \rightarrow 0, \hspace{5mm} \textrm{\rm as} \hspace{5mm} \sigma_0 \rightarrow 0, $$
and
$$ \mathbb{P}\Bigl( \tau^{-}_{\gamma}(\sigma_0) < T, \tilde \mu^{\sigma_0}_{\tau_{\epsilon_0}(\tilde \mu^{\sigma_0})}
 = - \epsilon_{0} \Bigr) \rightarrow 0, \hspace{5mm} \textrm{\rm as} \hspace{5mm} \sigma_0 \rightarrow 0.$$
\end{prop}

\begin{proof}
Consider $\sigma_0 \in (0,1)$
and call 
$\tilde t_{0}$ an infinitesimal function as in the statement of Proposition 
\ref{transitionpoint}. By 
Proposition 
\ref{transitionpoint}, by 
assumption (A2) in 
Subsection \ref{subse:5:2}
 and by symmetry, it suffices to prove that 
$$  \mathbb{P}\Bigl( \tau^{+}_{\gamma}(\sigma_0) < T, 
 \tau_{\epsilon_0}\bigl(\tilde \mu^{\sigma_0}\bigr) \leq \tilde{t_0}, 
 \tilde \mu^{\sigma_0}_{\tau_{\epsilon_0}(\tilde \mu^{\sigma_0})}
 = \epsilon_{0}, 
 \sup_{t \in [0,T]} \vert \triangle_{t}^{\sigma_{0}}
 \vert \leq \frac{\gamma}{4}
\Bigr) \rightarrow 0, \hspace{5mm} \textrm{\rm as} \hspace{5mm} \sigma_0 \rightarrow 0. $$

\vspace{4pt}

\textit{First step.} 
%
%
%
Throughout the step, we 
work on the event 
$\{
\tau^{+}_{\gamma}(\sigma_0) < T, 
 \tau_{\epsilon_0}\bigl(\tilde \mu^{\sigma_0}\bigr) \leq \tilde{t_0},
  \tilde \mu^{\sigma_0}_{\tau_{\epsilon_0}(\tilde \mu^{\sigma_0})} 
= \epsilon_{0},
 \sup_{t \in [0,T]} \vert \triangle_{t}^{\sigma_{0}}
 \vert \leq \gamma/4
 \}$. So, 
for $\sigma_{0}$ small enough such that $\tilde t_{0} \leq \delta/2$
and $L(\sigma_{0}) \geq 2$, 
we deduce from Lemma 
\ref{increasingbounds} that, for all $ t \in [ \tau_{\epsilon_0}(\tilde \mu^{\sigma_0}),\tau^{+}_{\gamma}({\sigma_0})],$
\begin{align*}
 \tilde \mu^{\sigma_0}_t  &\geq   \sigma_0^2 L(\sigma_{0})   + (1- \gamma) {\mathbf 1}_{t \geq \delta /2}\int_{\delta/2}^t w^{-2}_s ds
> \sigma_0^2 L(\sigma_{0})   + \bigl( r_{\delta} - r_{t} 
\bigr)_{+}
+ \frac12 \bigl( r_{\delta/2} - r_{t \wedge \delta} \bigr)_{+}.
\end{align*}
We deduce the following two things. First,  
\begin{equation}
\label{1}
\begin{split}
&\theta\bigl(t,\tilde \mu^{\sigma_0}_t\bigr) = -\textrm{\rm sign}(\tilde \mu^{\sigma_0}_t) = -1, \quad \forall t \in [ \tau_{\epsilon_0}(\tilde \mu^{\sigma_0}),\tau^{+}_{\gamma}({\sigma_0})].
\end{split}
\end{equation}
Moreover, for all $ t \in [\tau_{\epsilon_0}(\tilde \mu^{\sigma_0}),\tau^{+}_{\gamma}({\sigma_0})]$,
 the point $\tilde \mu^{\sigma_0}_t$ lies in the domain of application of 
 assumption 
(A1) in 
Subsection \ref{subse:5:2}, which yields
\begin{equation}
\label{2}
\bigl|\tilde \Psi\bigl(t,\tilde \mu^{\sigma_0}_t, \sigma_0\bigr)\bigr| \leq  \gamma / 4, 
\quad 
 t \in [\tau_{\epsilon_0}(\tilde \mu^{\sigma_0}),\tau^{+}_{\gamma}({\sigma_0})],
\end{equation}
whenever $\sigma_0>0$ is sufficiently small (uniformly in $t$), and where we let 
$\tilde \Psi(t,x,\sigma_{0}) = \tilde \theta^{\sigma_{0}}(t,x) - \theta(t,x)$.

Hence, for sufficiently small $\sigma_0 >0$ such that claims (\ref{1}) and (\ref{2}) both hold, we observe 
from (A2) that, for all  $t \in [
\tau_{\epsilon_0}(\tilde \mu^{\sigma_0}),\tau^{+}_{\gamma}({\sigma_0})]$, 
\begin{align*}
 \tilde \mu^{\sigma_0}_t &\geq 
 \tilde \mu^{\sigma_0}_{\tau_{\epsilon_0}(\tilde \mu^{\sigma_0})}
  - \int_{
\tau_{\epsilon_0}(\tilde \mu^{\sigma_0})}^t
w^{-2}_s \Bigl(
  \theta\bigl(s,\tilde \mu^{\sigma_0}_s\bigr)  + \tilde \Psi\bigl(s,\tilde \mu^{\sigma_0}_s, \sigma_0\bigr) - \triangle_{s}^{\sigma_{0}} \Bigr)  ds 
    + \sigma_0 \int_{
\tau_{\epsilon_0}(\tilde \mu^{\sigma_0})}^t w^{-1}_s dB_s \\
 &\geq  \epsilon_{0}  + (1- \gamma /2) \int_{\tau_{\epsilon_0}(\tilde \mu^{\sigma_0})}^t w^{-2}_s ds + \sigma_0 \int_{\tau_{\epsilon_0}(\tilde \mu^{\sigma_0})}^t w^{-1}_s dB_s.
\end{align*}

\textit{Second step.}
We work on the same event as in the first step. 
To simplify notation, let us write $\tau = \tau^{+}_{\gamma}( {\sigma_0})$ and observe that
(recall that we assumed $\tau <T$)
 $$\tilde \mu^{\sigma_0}_{\tau} =  \sigma_0^2 L(\sigma_{0})   + (1- \gamma)\int_{
 \tau_{\epsilon_0}(\tilde \mu^{\sigma_0})
}^{\tau} w^{-2}_s ds. $$ 
By the conclusion of the previous step, 
we deduce that there exists $ r \in [\tau_{\epsilon_0}(\tilde \mu^{\sigma_0}), T]$  such that 
$$ 
\sigma_0^2 L(\sigma_{0})   + (1- \gamma)\int_{
 \tau_{\epsilon_0}(\tilde \mu^{\sigma_0})
}^{r} w^{-2}_s ds  \geq \epsilon_{0} + (1- \gamma /2) \int_{\tau_{\epsilon_0}(\tilde \mu^{\sigma_0})}^r w^{-2}_s ds + \sigma_0 \int_{\tau_{\epsilon_0}(\tilde \mu^{\sigma_0})}^r w^{-1}_s dB_s.$$  
Therefore,  on the same event as in the first step, 
there exists $ r \in [\tau_{\epsilon_0}(\tilde \mu^{\sigma_0}), T]$  such that  $$  \hspace{2mm}   \sigma_0^{-1} \int_{\tau_{\epsilon_0}(\tilde \mu^{\sigma_0})}^r w^{-1}_s dB_s+  \sigma_0^{-2} (\gamma  /2) \int_{\tau_{\epsilon_0}(\tilde \mu^{\sigma_0})}^r w^{-2}_s ds  \leq  L(\sigma_{0}) \bigl( 1 - L(\sigma_{0}) \bigr) =: v(\sigma_0). $$ 
We deduce that the event 
$\{
\tau^{+}_{\gamma}(\sigma_0) < T, 
 \tau_{\epsilon_0}\bigl(\tilde \mu^{\sigma_0}\bigr) \leq \tilde{t_0},
  \tilde \mu^{\sigma_0}_{\tau_{\epsilon_0}(\tilde \mu^{\sigma_0})} 
= \epsilon_{0},
 \sup_{t \in [0,T]} \vert \triangle_{t}^{\sigma_{0}}
 \vert \leq \gamma/4
 \}$ is included in the event
\begin{equation*}
\biggl\{ \sup_{r \in [\tau_{\epsilon_0}(\tilde \mu^{\sigma_0}), T]}
\exp \biggl( - 
\gamma \sigma_0^{-1} \int_{\tau_{\epsilon_0}(\tilde \mu^{\sigma_0})}^r w^{-1}_s dB_s -  
\frac12\sigma_0^{-2} \gamma^2 \int_{\tau_{\epsilon_0}(\tilde \mu^{\sigma_0})}^r w^{-2}_s ds
\biggr) \geq \exp \bigl( - \gamma v(\sigma_{0}) \bigr) \biggr\}.
\end{equation*}
\vspace{4pt}

\textit{Third step.} It remains to prove that 
\begin{equation*}
{\mathbb P}\biggl( \sup_{r \in [\tau_{\epsilon_0}(\tilde \mu^{\sigma_0}), T]}
\exp \biggl( - 
\gamma \sigma_0^{-1} \int_{\tau_{\epsilon_0}(\tilde \mu^{\sigma_0})}^r w^{-1}_s dB_s -  
\frac12\sigma_0^{-2} \gamma^2 \int_{\tau_{\epsilon_0}(\tilde \mu^{\sigma_0})}^r w^{-2}_s ds
\biggr) \geq \exp \bigl( - \gamma v(\sigma_{0}) \bigr) \biggr) \rightarrow 0,
\end{equation*}
as $\sigma_{0}$ tends to $0$. 
This is a simple consequence of Doob's maximal inequality for the 
martingale 
\begin{equation*}
\biggl( 
\exp \biggl( - 
\gamma \sigma_0^{-1} \int_{\tau_{\epsilon_0}(\tilde \mu^{\sigma_0})}^{t \vee \tau_{\epsilon_0}(\tilde \mu^{\sigma_0})} w^{-1}_s dB_s -  
\frac12\sigma_0^{-2} \gamma^2 \int_{\tau_{\epsilon_0}(\tilde \mu^{\sigma_0})}^{t \vee \tau_{\epsilon_0}(\tilde \mu^{\sigma_0})} w^{-2}_s ds
\biggr)
\biggr)_{t \in [0,T]},
\end{equation*}
and of the fact that $v(\sigma_0) \rightarrow - \infty$. 
\end{proof}

\subsection{Conclusion of the proof of
Theorem
\ref{thm:general}
}

\begin{proof}
\textit{First step.}
It is easily checked that the sequence $({\mathbb P} \circ 
(\tilde \mu^{\sigma_{0}}_{t})_{t \in [0,T]}^{-1})_{\sigma_{0}>0}$ is tight on ${\mathcal C}([0,T];{\mathbb R})$. Also, by (A2), we deduce that, for any limiting point 
${\mathbb P}^{\infty}$, the canonical process  
$(\psi_{t})_{t \in [0,T]}$ on ${\mathcal C}([0,T];{\mathbb R})$ satisfies
$\vert \psi_{t} \vert \leq k_{t}$ for all $t \in [0,T]$, with probability 1 under ${\mathbb P}^{\infty}$.
\vskip 4pt

\textit{Second step.}
In order to proceed further, we need new notation. 
For $t \in [0,\delta/2]$, we let 
$F(t):=\{ \psi \in {\mathcal C}([0,T];{\mathbb R}) : \psi_{s} \geq k_{s} - k_{t}, \ s \in [t,T]\}$ and, for any $\epsilon >0$, 
$F^{\epsilon}(t): =\{ \psi \in {\mathcal C}([0,T];{\mathbb R}) : 
\psi_{s} \geq k_{s} - k_{t} - \epsilon, \ s \in [t,T]\}$.
Obviously, $F(t)$ and $F^{\epsilon}(t)$ are closed subsets of 
${\mathcal C}([0,T];{\mathbb R})$.

We claim that for any limit point ${\mathbb P}^\infty$ 
of the sequence $({\mathbb P} \circ 
(\tilde \mu^{\sigma_{0}}_{s})_{s \in [0,T]}^{-1})_{\sigma_{0}>0}$,  
it holds
\begin{equation*}
{\mathbb P}^\infty \Bigl( \bigl\{
\psi \in {\mathcal C}([0,T];{\mathbb R}) : 
 \psi \in 
F(0) \bigr\} \cup
\bigl\{
\psi \in {\mathcal C}([0,T];{\mathbb R}) : 
 - \psi \in 
F(0) \bigr\}
\Bigr) = 1. 
\end{equation*}
Below, we merely write
$F(0) \cup (-F(0))$ for $
\{
\psi \in {\mathcal C}([0,T];{\mathbb R}) : 
 \psi \in 
F(0) \} \cup \{
\psi \in {\mathcal C}([0,T];{\mathbb R}) : 
 - \psi \in 
F(0)\}$.

We apply Proposition \ref{boundasymp}. 
For  a given $\gamma >0$, it says that 
$$ \lim_{\sigma_{0}\rightarrow 0} 
 \mathbb{P}
\Bigl( \bigl\{ 
 \tau^{+}_{\gamma}(\sigma_0) = T,\tilde \mu^{\sigma_0}_{\tau_{\epsilon_0}(\tilde \mu^{\sigma_0})}
 = \epsilon_{0} 
 \bigr\} 
 \cup 
\bigl\{ 
 \tau^{-}_{\gamma}(\sigma_0) = T,\tilde \mu^{\sigma_0}_{\tau_{\epsilon_0}(\tilde \mu^{\sigma_0})}
 = -\epsilon_{0} 
\bigr\}
 \Bigr)
 = 1.$$ 
Take $t \in (0,\delta/2)$ and fix $\epsilon >0$. 
On the event 
$\{ \tau_{\epsilon_0}(\tilde \mu^{\sigma_0}) \leq t \}  \cap \{ 
 \tau^{+}_{\gamma}(\sigma_0) = T,\tilde \mu^{\sigma_0}_{\tau_{\epsilon_0}(\tilde \mu^{\sigma_0})}
 = \epsilon_{0} 
 \}$, we have
 \begin{equation*}
 \tilde{\mu}^{\sigma_{0}}_{s} \geq (1- \gamma) 
\bigl( 
 k_{s} - k_{t} \bigr), \qquad s \in [t,T],
\end{equation*} 
and, for $\gamma$ small enough,
 \begin{equation*}
 \tilde{\mu}^{\sigma_{0}}_{s} \geq 
 k_{s} - k_{t} - \epsilon, \qquad s \in [t,T],
\end{equation*} 
Therefore, on the event  
$\{ \tau_{\epsilon_0}(\tilde \mu^{\sigma_0}) \leq t \}  \cap \{ 
 \tau^{+}_{\gamma}(\sigma_0) = T,\tilde \mu^{\sigma_0}_{\tau_{\epsilon_0}(\tilde \mu^{\sigma_0})}
 = \epsilon_{0} 
 \}$, 
 $(\tilde{\mu}^{\sigma_{0}}_{s})_{s \in [0,T]} \in 
F^{\epsilon}(t)$.

Then, Proposition \ref{transitionpoint} says that 
\begin{equation*}
\lim_{\sigma_{0} \rightarrow 0} {\mathbb P} \Bigl( (\tilde{\mu}_{s}^{\sigma_{0}})_{s \in [0,T]}
\in 
F^{\epsilon}(t) \cup (-F^{\epsilon}(t))
\Bigr) = 1. 
\end{equation*}
Since 
$F^{\epsilon}(t) \cup (-F^{\varepsilon}(t))$
is closed, 
we get, 
by the portmanteau theorem, that, for all $t \in [0,\delta/2]$ and $\epsilon >0$, 
\begin{equation*}
{\mathbb P}^{\infty} \bigl( F^{\epsilon}(t) \cup (-F^{\epsilon}(t)) \bigr) = 1.
\end{equation*} 
Intersecting over all the positive and rational reals $\epsilon$, we get:
\begin{equation*}
{\mathbb P}^{\infty} \bigl( F(t) \cup (-F(t)) \bigr) = 1.
\end{equation*} 
Intersecting over all the rational reals $t \in [0,\delta/2]$, we deduce the announced claim.
\vskip 4pt

\textit{Conclusion.} By the first and second steps, the canonical process $(\psi_{t})_{t \in [0,T]}$
on ${\mathcal C}([0,T];{\mathbb R})$ must satisfy
\begin{equation*}
{\mathbb P}^{\infty} \Bigl( 
\{ \psi_{t} =k_t, \quad t \in [0,T] \}
\cup 
\{ \psi_{t} =-k_t, \quad t \in [0,T] \}
\Bigr) = 1.
\end{equation*}
By (A3), $\psi$
and $-\psi$ have the same law under
${\mathbb P}^{\infty}$. We deduce that
\begin{equation*}
{\mathbb P}^{\infty} \bigl(   \psi_{t} =k_t, \quad t \in [0,T]  
\bigr)
=
{\mathbb P}^{\infty} \bigl(   \psi_{t} =-k_t, \quad t \in [0,T] 
\bigr) = \frac12,
\end{equation*}
which completes the proof.
\end{proof}

\section{$N$-player limit selection}
\label{se:N-player}

We now come to the last method of selection. As shown in Subsection \ref{subse:Nplayer_game}, we can indeed associate with our particular LQ-MFG a game with a finite number of players 
and then address the asymptotic form of the equilibria (if any) as the number of players tends to $\infty$. In fact, 
the connection between 
mean-field games and 
games 
with finitely many players is a major question in the theory of mean-field games, see for instance the references \cite{Lions,Cardaliaguet,CarmonaDelarue_book_I,CarmonaDelarue_book_II,Fischer,Lacker_limits,Lacker2}. 

Below, 
we prove that, for a finite number $N$ of players, 
Nash equilibria (if any) solve a forward-backward stochastic particle system. The goal is thus to address the asymptotic form, under the limit $N \rightarrow \infty$, 
of the solution to this particle system and to see which equilibria of the LQ-MFG \eqref{fbsde2:sigma0=0} are charged by the weak limits (if any).

Basically, the main result that we show in this section is that the equilibria that are selected in this way are the same as those 
obtained in the previous section. 

\subsection{The associated $N$-players games}
\label{subse:Nplayer_game}

In this paragraph, we formulate the version with finitely many players of the mean-field game we have been considering so far. 
As already explained in introduction, the fact that the mean-field game has a counterpart in the form of a stochastic differential game 
with a finite number of players is not a big surprise: This connection is pretty standard and, in fact, it 
is the basis of 
the whole theory of mean-field games, see the aforementioned references. 

The striking fact in the game with finitely many players we address below is that each player is driven by its own 
Brownian motion. In other words, noises are independent; they are said to be idiosyncratic. 

So, for the description of the game, we consider an integer $N \in \mathbb{N}^*$, which stands for the number of players in the game. 
Then, for the same time horizon as before, we call $(W^i_t \hspace{2mm}; i =1,...,N)_{t \in [0,T]}$ a collection of 
$N$ independent one dimensional Brownian motions defined on a (common) complete filtered probability space $(\Omega,  {\mathcal{F}}, \mathbb{P})$. We call 
$(\hat{\mathcal{F}_t})_{t\in [0,T]}$ the usual augmentation of the filtration generated by 
$((W^i_t)_{t \in [0,T]} \hspace{2mm}; {i =1,...,N})$.

Also $\sigma >0$ and $\kappa \in \mathbb{R}$ are the same constants as in the system \eqref{constraint2}, and $ g : \mathbb{R} \rightarrow \mathbb{R} $ 
is the Lipschitz continuous and bounded function we defined earlier. Importantly, $\sigma_{0}$ is $0$ in this paragraph: 
There is no common noise; but somehow, we show below that 
there is an \textit{intrinsic common noise} of variance $\sigma_{0}= \sigma N^{-1/2}$ in the system. 
We will insist repeatedly on this fact which will serve us as a guideline. 
\vskip 4pt

\subsubsection*{Formulation of the game}

For all $i=1,...,N$, the evolution of $i^{th}$ player's state during the game is described by the real-valued process $(X^i_t)_{t \in [0,T]}$. 
Noticeably, player $i$ sees the other players through an aggregate quantity, which is here given by the average of the states of 
 all these other players, namely
 $$\mu^{i,N}_t : = \frac{1}{N-1}\sum_{j \neq i}^{N} X^j_t, \hspace{2mm} \forall t \in [0,T].$$
Of course, 
the fact that interactions are designed in such a way is the cornerstone for explaining the mean-field structure 
we addressed in Scheme 
\ref{MFG-problem}. 
\begin{rem}
In some of the articles on the subject, authors include in the definition of the empirical measure the own state of player 
$i$, in which case $\mu^{i,N}_{t}$ becomes independent of $i$ and writes
$$\mu^{N}_t = \frac{1}{N}\sum_{j=1}^{N} X^j_t.$$
As explained in \cite[Chapter 6]{CarmonaDelarue_book_II}, 
the limiting game should be the same. Still, we here work with the first 
form of the empirical measure as it is more convenient for our own purposes. 
\end{rem}

Player $i$ has the following dynamics:
\begin{equation}
\label{constraint0}
\begin{cases}
dX^i_t = [ \kappa X^i_t + \alpha^i_t] dt + \sigma dW^i_t , \hspace{2mm} \forall t \in [0,T],
\\ 
X^i_0 = 0,
\end{cases}
\end{equation}
where $(\alpha_{t}^i)_{t \in [0,T]}$ is a control process 
 belonging to the space 
$\hat{\mathcal{H}^2}$ of $(\hat{\mathcal{F}_t})_{t \in [0,T]}$-progressively measurable processes 
$(\alpha_{t})_{t \in [0,T]}$ 
satisfying $$  \mathbb{E} \bigg[ \int_0^T |\alpha_t|^2 dt \bigg] < \infty.$$   
Given the tuple of controls $(\alpha^1_{t},\cdots,\alpha_{t}^N)_{t \in [0,T]}$, we associate with player 
$i$ the following cost functional:
\begin{align}
\label{costfxn0}
J^i(\alpha^1,\cdots,\alpha^i,\cdots,\alpha^N)  & :=  \mathbb{E} \Bigg[ \int_0^T \frac{1}{2} \bigg[( \alpha^i_t)^2 + \big( X^i_t \big)^2 \bigg] dt + \frac{1}{2} \bigg( X^i_T + g ( \mu^{i,N}_T ) \bigg)^2 \Bigg].
\end{align}
We then recall the following standard definition:

\begin{defi}
We call $(\alpha^1_{t},\cdots,\alpha_{t}^N)_{t \in [0,T]}$ a Nash equilibrium if, for any
$i \in \{1,\cdots,N\}$
and any 
 other 
process $(\beta_{t})_{t \in [0,T]}$,
\begin{equation*}
J^i(\alpha^1,\cdots,\alpha^{i-1},\beta,\alpha^{i+1},\cdots,\alpha^N)
\geq 
J^i(\alpha^1,\cdots,\alpha^{i-1},\alpha^i,\alpha^{i+1},\cdots,\alpha^N).
\end{equation*}
\end{defi}
In other words, a Nash equilibrium is a consensus between the players: None of them can be better off by deviating unilaterally from the consensus. 

It must be emphasized that the definition given above is restricted to so-called equilibria over controls in
 open-loop form: 
When player $i$ changes her/his own strategy, the others keep playing
the same realizations of  
$((\alpha^j_{t})_{t \in [0,T]} \hspace{2mm} ; {j \not = i})$. This is contrast with equilibria over controls in Markovian closed loop form, 
which are addressed in the PDE literature:
Equilibria over controls in Markovian closed loop form are in the form 
$(\bar \alpha^j(t,X^1_{t},\cdots,X^N_{t}))_{t \in [0,T]}$, for $j = 1,\cdots,N$,
for functions $\bar \alpha^j : [0,T] \times {\mathbb R}^N \rightarrow {\mathbb R}$;
whenever player $i$ deviates, she/he chooses another feedback function  
$\bar \beta : [0,T] \times \RR^N \rightarrow \RR$ instead of 
$\bar \alpha^i$ while the others keep using 
$\bar \alpha^j$; still, as the values of the state process $(X^i_{t})_{t \in [0,T]}$ change,
the realizations of the control processes
$((\bar \alpha^j(t,X^1_{t},\cdots,X^N_{t}))_{t \in [0,T]} \hspace{2mm} ;{j \not = i})$
change as well. We refer to \cite[Chapter 2]{CarmonaDelarue_book_I} 
for a review. 

We shall not address the case of equilibria over controls in Markovian closed loop form in the text, but this could make sense as well.

\subsubsection*{First order condition}

Similar to \eqref{fbsde2} (for mean-field games), we can write down a first order condition 
for the Nash equilibria of the game 
\eqref{constraint0}--\eqref{costfxn0}
 in the form of a system of forward-backward stochastic differential equations. 
This is the cornerstone of our selection result. Again, we refer 
to \cite[Chapter 2]{CarmonaDelarue_book_I} 
for details on the derivation of this forward-backward system.

The first order condition writes as follows. Any Nash equilibrium to the associated $N$-player game 
 \eqref{constraint0}--\eqref{costfxn0}
 is in the set of solutions of the following system of forward-backward SDEs: 
\begin{equation}
\label{fbsdeN}
\begin{cases}
dX^{i}_t = \bigl[ \kappa X^{i}_t - Y^{i}_t \bigr]dt + \sigma dW^{i}_t, 
\quad 
\forall t \in [0,T],
\quad X^{i}_0 = 0.  
\vspace{2pt}
\\ 
dY^{i}_t  = \bigl[ -X^{i}_t - \kappa Y^{i}_t \bigr]dt + \sum_{k=1}^{N} Z^{i,k}_t dW^{k}_t, 
\quad \forall t \in [0,T],
\quad Y^{i}_T = X^{i}_T + g(\mu^{i,N}_T),
\vspace{2pt}
\\ 
i \in \{1,2,...,N\}. 
\end{cases}
\end{equation}
In other words, the state processes of any Nash equilibrium must coincide 
with the forward paths of some solution to the above system. 

Below, we do not discuss whether \eqref{fbsdeN}
is a sufficient condition or not. Usually, it is known to be sufficient in the case when 
the coefficients of the cost functional 
\eqref{costfxn0} are convex, but the latter is not true here. 

Following our strategy for solving the system \eqref{fbsde2}, 
we search for a solution $(X^i_t, Y^{i}_t, Z^{i,k}_t \hspace{2mm}; i , k = 1, 2, ..., N)_{t\in[0,T]}$ to the FBSDE (\ref{fbsdeN}) in
the form  $Y^{i}_t = \eta_t X^i_t + V^i_t$ for all $i = 1,2,...,N$ and $t \in [0,T],$ where $$ dV^i_t = \chi^i_t dt  + \sum_{k=1}^{N} z^{i,k}_t dW^k_t, \quad V^i_T = g(\mu^{i,N}_T).$$

In fact, we can show that  there is a solution $(X^i_t, Y^{i}_t, Z^{i,k}_t \hspace{2mm}; i , k = 1, 2, ..., N)_{t\in[0,T]}$ to FBSDE (\ref{fbsdeN}) if and only if we can construct a solution $(X^i_t, V^{i}_t, z^{i,k}_t \hspace{2mm}; i , k = 1, 2, ..., N)_{t\in[0,T]}$ to FBSDE (\ref{fbsdeNext}) below: 
\begin{equation}
\label{fbsdeNext}
\begin{cases}
dX^{i}_t = \bigl[ (\kappa-\eta_t) X^{i}_t - V^{i}_t \bigr]dt + \sigma dW^{i}_t, 
\quad
\forall t \in [0,T],
\quad X^{i}_0 = 0.  
\vspace{2pt}
\\ 
dV^{i}_t  =  - (\kappa-\eta_t) V^{i}_t dt + \sum_{k=1}^{N} z^{i,k}_t dW^{k}_t, 
\quad
\forall t \in [0,T],
\quad V^{i}_T =  g(\mu^{i,N}_T), 
\vspace{2pt}
\\
i \in \{1,2,...,N\}, 
\vspace{5pt} 
\\ 
\frac{ d\eta_t}{dt} =  \eta^2_t -2\kappa \eta_t - 1, 
\quad
\forall t \in [0,T],
\quad \eta_T = 1.
\end{cases}
\end{equation}
The connection between \eqref{fbsdeN} to \eqref{fbsdeNext} 
is given by the change of variable $(Y_{t}^{i} = \eta_{t} X_{t}^i + V_{t}^i \hspace{2mm} ; i =1,...,N)_{t \in [0,T]}$. In fact, by a new (straightforward) change of variable in the forward component, 
we can even remove the $X^i$ dependence in the drift of the forward equation. By \cite{Delarue02},
we deduce that \eqref{fbsdeNext} is uniquely solvable. Hence, 
\eqref{fbsdeN} is uniquely solvable as well as. Although this does not show the existence of a Nash equilibrium to the $N$-player game, this shows that the first order condition is always uniquely satisfied, which suffices for our purposes: Below, we investigate the asymptotic behavior of the solution to 
\eqref{fbsdeN}  as $N$ tends to $+\infty$. 

Recall now that $w_t = \exp( - \int_t^T (\kappa-\eta_s) ds ) $ and define the rescaled average players' states (over all the players): 
$$\tilde \mu^N_t = \frac{w_t^{-1}}{N} \sum_{i=1}^{N} X^i_t, \quad v^N_t = \frac{w_t}{N} \sum_{i=1}^{N} V^i_t, \quad \text{for
 all} \ t \in [0,T]. $$
Therefore, when there exists a Nash equilibrium to the associated $N$-player game, the process $(\tilde \mu_t^N,v_t^N)_{t\in[0,T]}$ is a solution to (use the fact that $w_{T}^{-1}=1$)

\begin{equation}
\label{fbsdeNew}
\begin{cases}
\forall t \in [0,T],
\vspace{2pt}
\\ 
d\tilde \mu^N_t =  -w^{-2}_t v^N_t dt + \frac{\sigma}{N} \sum_{i=1}^{N} w^{-1}_t dW^{i}_t, \quad \tilde \mu^N_0 = 0,  
\vspace{2pt}
\\ 
(v^N_t)_{t\in[0,T]} \quad \text{is a continuous martingale}, \quad v^N_T =  \frac{1}{N} \sum_{i=1}^{N} g(\tilde \mu^{i,N}_T),
\end{cases}
\end{equation}
where
\begin{equation*}
\tilde \mu^{i,N}_t
= \frac{w_{t}^{-1}}{N-1} \sum_{j \not =i} X_{t}^j.
\end{equation*}

\subsubsection*{Main statement}

Here is now our main statement:
\begin{thm}
\label{mainN} 
Consider $\big( k_t := \int^t_0 w_s^{-2}ds \big)_{t \in [0,T]} \in \mathcal{C}([0,T];\mathbb{R})$.
The sequence $({\mathbb P} \circ 
(\tilde \mu^{N}_{t})_{0 \le t \le T}^{-1})_{N \geq 1}$ converges, as $N \rightarrow \infty$, to 
\begin{equation*}
\frac12 \delta_{(k_t)_{t \in [0,T]}}
+\frac12 \delta_{(-k_t)_{t \in [0,T]}}.
\end{equation*}


\end{thm}

Of course, it says that the equilibria that are selected in 
Proposition 
\ref{prop:sol:sigma0=0}
are the same as those selected in Theorem \ref{mainzeronoise}. 
\subsection{Approximate decoupling field}

The problem with \eqref{fbsdeNew} is that the terminal condition is not
in the form of a function of $\mu^N_{T}$. Still, what we expect is that 
the solution to \eqref{fbsdeNew} should get closer and closer (as $N$ tends to $\infty$) to the solution of the 
same system but with the terminal boundary condition
\begin{equation*}
v^N_T =  \frac{1}{N} \sum_{i=1}^{N} g(\tilde \mu^{N}_T) = g(\tilde \mu^N_{T}),
\end{equation*}
namely (we put a hat on the symbols to distinguish from \eqref{fbsdeNew})
\begin{equation}
\label{fbsdeNew:approx}
\begin{cases}
\forall t \in [0,T],
\vspace{2pt}
\\ 
d\hat \mu^N_t =  -w^{-2}_t \hat v^N_t dt + \frac{\sigma}{N} \sum_{i=1}^{N} w^{-1}_t dW^{i}_t, \quad \hat \mu^N_0 = 0,  
\vspace{2pt}
\\ 
(\hat v^N_t)_{t\in[0,T]} \quad \text{is a continuous martingale}, \quad \hat v^N_T =  g(\hat \mu^{N}_T).
\end{cases}
\end{equation}
Letting
\begin{equation}
\label{eq:B:sumi:Wi}
 B_{t} = \frac{1}{\sqrt{N}} \sum_{i=1}^N W_{t}^i, \quad t \in [0,T],
\end{equation}
we then recover \eqref{fbsde2} and  $(B_{t})_{t \in [0,T]}$, which is a Brownian motion with respect to $(\hat{\mathcal{F}_t})_{t\in [0,T]}$, plays the role of a common noise with intensity 
$\sigma_{0}=\sigma N^{-1/2}$! (This explains why we require $\sigma >0$ in this section.) 

However, this does not work so easily since \eqref{fbsdeNew:approx}
and 
\eqref{fbsdeNew} do not coincide. So, we pay some price below to estimate the distance between the solutions 
of the two systems. 

\subsubsection*{Comparison argument}

The main difficulty to compare \eqref{fbsdeNew:approx}
and 
\eqref{fbsdeNew} is the fact that, as $N$ tends to $\infty$, the system \eqref{fbsdeNew:approx} becomes ill-posed. So, we cannot expect for 
robust stability properties, uniformly in the parameter $N$, to estimate the difference between the solutions of the two equations.  

The strategy we use below is based upon a comparison principle. As shown by 
Lemma \ref{comparePDEs}, a form of comparison should be indeed in force independently of the value of 
$N$. 
\vskip 5pt

In order to put things in order, we recall  that $r_\delta = \int_\delta^T w_s^{-2}ds$ and we consider a sequence of positive real numbers $(\gamma_N)_{N \geq 1} \subseteq (0, r_\delta / 2)$ such that  $\gamma_N \rightarrow 0$
as $N \rightarrow +\infty $. We then define the Lipschitz continuous non-increasing function 
\begin{equation}
\label{eq:tilde:g}
\tilde{g}^N(x) : = 
\begin{cases}
g(x), \quad   &\text{if} \quad  x \leq r_\delta - 2\gamma_N,
\vspace{2pt} \\
g(r_\delta - 2\gamma_N) \quad   &\text{if} \quad r_\delta- 2 \gamma_N \leq x \leq r_\delta-\gamma_N,
\vspace{2pt} \\
g(x-\gamma_N) \quad   &\text{if} \quad  r_\delta - \gamma_N \leq x \leq r_\delta + \gamma_N,
\vspace{2pt} \\
g(x) = g(r_{\delta}) \quad   &\text{if} \quad  x \geq  r_\delta+\gamma_N.
\end{cases}
\end{equation}
We then have the following lemma:

\begin{lem}
The functions
$g$ and  
$\tilde{g}^N$ satisfy
$$ \tilde{g}^N \geq g \quad \text{and} \quad \int_{-\infty}^{+\infty} (\tilde{g}^N - g)(x) dx = 2 \frac{\gamma^2_N}{r_{\delta}}.$$
\end{lem}

\begin{proof}
The inequality $\tilde{g}^N \geq g$ is a consequence of the fact that $g$ is non-increasing. As for the second part of 
the statement, we have
\begin{equation*}
\begin{split}
 \int_{-\infty}^{+\infty} (\tilde{g}^N - g)(x) dx
 &= \int_{r_{\delta} - 2\gamma_{N}}^{r_{\delta}+ \gamma_{N}}
 (\tilde{g}^N - g )(x) dx
 \\
 &= \int_{r_{\delta} - 2\gamma_{N}}^{r_{\delta} - \gamma_{N}} 
  \frac{x - r_{\delta}+2\gamma_{N}}{r_{\delta}}  dx
 +
  \int_{r_{\delta}-\gamma_{N}}^{r_{\delta}}
  \frac{\gamma_{N}}{r_{\delta}}  dx
+
 \int_{r_{\delta}}^{r_{\delta}+ \gamma_{N}}
  \frac{r_{\delta}-x + \gamma_{N}}{r_{\delta}}  dx
  \\
&  = 2 \frac{\gamma^2_N}{r_{\delta}}.
\end{split}
\end{equation*}
\end{proof}

Following 
\eqref{PDE}, we can associate with 
$g$ and $\tilde{g}^N$
the functions $ \theta^N$ and $\tilde{\theta}^N$ in $\mathcal{C}^{1,2}([0,T); \mathbb{R}) \cap {\mathcal C}([0,T] \times {\mathbb R};{\mathbb R})$, classical solutions to 
\begin{align}
\label{PDE3}
\begin{cases}
\partial_t \tilde{\theta}^{N}_t - w^{-2}_t \tilde{\theta}^{N} \partial_x \tilde{\theta}^{N} + \frac{1}{2} (\frac{\sigma^2}{N}) w^{-2}_t \partial^{2}_{xx}\tilde{\theta}^{N}  = 0, \hspace{3mm} &\forall t \in [0,T), \hspace{2mm} \forall x \in \mathbb{R}, 
\\
\partial_t \theta^{N}_t - w^{-2}_t \theta^{N} \partial_x \theta^{N} + \frac{1}{2} (\frac{\sigma^2}{N}) w^{-2}_t \partial^{2}_{xx} \theta^{N}  = 0, \hspace{3mm} &\forall t \in [0,T), \hspace{2mm} \forall x \in \mathbb{R},
\\
\tilde{\theta}^{N}(T, x) = \tilde{g}^N(x), \hspace{2mm} \theta^{N}(T, x)  = g(x), \hspace{2mm} &\forall x \in \mathbb{R}. 
\end{cases}
\end{align}
Then, thanks to Lemma \ref{comparePDEs}, we have
$\tilde \theta^N \geq \theta^N$ and
\begin{equation}
\label{eq:2612:1}
0 \leq  \int_{-\infty}^{+\infty} (\tilde{\theta}^{N} - \theta^{N})(t,x) dx \leq  \int_{-\infty}^{+\infty} (\tilde{g}^N - g)(t,x) dx = 
\frac{2 \gamma^2_N}{r_{\delta}}, \quad \textrm{\rm for all} \quad t \in [0,T].
\end{equation}

\subsubsection*{Comparison between $v^N$ and $\tilde \theta^N$}

Recall the process $(k_t = \int_0^t w_s^{-2} ds)_{t \in [0,T]}$. Here is our main comparison result:

\begin{lem}
\label{le:0}
Choose $(\gamma_{N})_{N \geq 1}$ 
in
\eqref{eq:tilde:g}
such that
 $${\gamma_{N}} \sqrt{N} \rightarrow +\infty \quad \text{as} \quad N \rightarrow +\infty.$$ 
 Then, with probability 1,
$$ -v^N_t \geq -\tilde{\theta}^N(t, \mu^N_t) +
\triangle^N_{t},  \quad \text{for all} \quad t \in [0,T], $$
where 
\begin{equation*}
\lim_{N \rightarrow +\infty} {\mathbb E} \Bigl[ \sup_{t \in [0,T]}
\bigl\vert \triangle^N_{t}
\bigr\vert^2 \Bigr] =0.
\end{equation*} 
\end{lem}

\begin{proof}
\textit{First step.}
For all $i = 1,2,...,N$, we can quantify the distance between $\tilde \mu^N_T$ and $\tilde \mu^{i,N}_T$ as follows: 
\begin{equation*}
|\tilde \mu^N_T - \tilde \mu^{i,N}_T| = \frac{1}{N} | \tilde X^i_T - \tilde \mu^{i,N}_T |,
\end{equation*}
where $(\tilde X_{t}^i := w_{t}^{-1} X_{t}^i,\tilde V_{t}^i := w_{t} V_{t}^i)_{t \in [0,T]}$ solves
\begin{equation*}
d \tilde{X}_{t}^i = - w_{t}^{-2} \tilde{V}_{t}^i dt + \sigma w_{t}^{-1} dW_{t}^i. 
\end{equation*}
Therefore,
\begin{equation*}
d \bigl( \tilde{X}_{t}^i - \tilde{\mu}^{i,N}_{t} \bigr) 
= - w_{t}^{-2} \Bigl( \tilde{V}_{t}^i - \frac1{N-1} \sum_{j \not = i}\tilde{V}_{t}^j \Bigr) dt + \sigma 
\Bigl( 
w_{t}^{-1} dW_{t}^i -  \frac{1}{N-1} \sum_{j \not= i} 
w_{t}^{-1} dW_{t}^j 
\Bigr). 
\end{equation*}
Obviously, $(\tilde V_{t}^i)_{t \in [0,T]}$ is a martingale with a terminal boundary condition that belongs 
to $[-1,1]$. Therefore, 
\begin{align*}
|\tilde \mu^N_T - \tilde \mu^{i,N}_T| & \leq \frac{1}{N} \Bigg[ 2 k_T + \sigma \bigg( \bigg| \int_0^T w_s^{-1} dW^i_s \bigg| + \frac{1}{N-1} \sum_{j\neq i}^{N} \bigg| \int_0^T w_s^{-1} dW^j_s \bigg| \bigg) \Bigg] \notag \\
& \leq \frac{2}{N} \Bigg[ k_T + \sigma \max_{i \in \{ 1,...,N \}} \bigg| \int_0^T w_s^{-1} dW^i_s \bigg|  \Bigg]. \notag \\
\end{align*}

\textit{Second step.}
By a straightforward application of It\^o's formula, we obtain 
\begin{equation*}
d\bigl( v^N_{t} - 
 \tilde{\theta}^N(t, \tilde \mu^N_t)
 \bigr) =  w_{t}^{-2} \partial_{x} \tilde \theta^N(t,\tilde \mu_{t}^N) 
 \bigl( 
 v^N_{t}
 - 
\tilde{\theta}^N(t, \tilde \mu^N_t)
 \bigr) 
 dt+
dM_{t},
\end{equation*}
where $(M_{t})_{t \in [0,T]}$ is a square integrable martingale. 
Therefore,
$$ v^N_t - \tilde{\theta}^N(t, \tilde \mu^N_t) = \mathbb{E} \Bigg[ \exp\bigg( -\int_t^T w_s^{-2} \partial_x \tilde \theta^N(s,\tilde \mu^N_s) ds \bigg) \bigg[v^N_T - \tilde{g}^N(\tilde \mu^N_T) \bigg] \Bigg| \hat{\mathcal{F}_t} \Bigg] \quad \text{for all} \quad t \in [0,T]. $$
Now, following
Lemma \ref{lem:gradient}, we know that there exists a positive constant $C$, independent of $N$, such that $$ |\partial_x \tilde{\theta}^N(s,\tilde \mu^N_s)| \leq C N \quad \text{for all} \quad s \in [0,T).$$
Therefore, by letting $A_N := \lbrace \max_{i \in \{ 1,...,N \}} | \int_0^T w_s^{-1} dW^i_s  | \leq C \sqrt{N} \rbrace$, we get
\begin{equation}
\label{vbound}
\begin{split}
 -v^N_t &\geq -\tilde{\theta}^N(t, \tilde \mu^N_t) + \triangle_t^N 
 \\
&\hspace{10pt} + \mathbb{E} \Bigg[
  \exp\bigg( -\int_t^T w_s^{-2} \partial_x \tilde \theta^N(s,\tilde \mu^N_s) ds \bigg) 
 \big[ \tilde{g}^N(\tilde \mu^N_T) - v^N_T \big] \mathbf{1}_{A_N} \Bigg| \hat{\mathcal{F}_t} \Bigg],
 \end{split}
\end{equation}
for all $t \in [0,T]$, 
with 
\begin{align}
|\triangle_t^N | & \leq \mathbb{E} \Bigg[ \exp(CN) \big\vert v^N_T - \tilde{g}^N(\tilde \mu^N_T) \big\vert \mathbf{1}_{A_N^{\complement}} \Bigg| \hat{\mathcal{F}_t} \Bigg] \notag \\
& \leq 2 \exp(CN) \mathbb{P} \biggl( \max_{i \in \{ 1,...,N \}} \biggl| \int_0^T w_s^{-1} dW^i_s  \biggr| > C \sqrt{N} 
\bigg|  \hat{\mathcal{F}_t} \biggr).
 \notag 
\end{align}
By standard Gaussian estimates, we notice that, for a constant $c$ depending on $k_{T}$, but independent of $C$,
\begin{align*} 
\max_{i \in \{ 1,...,N \}} \mathbb{P} \biggl(  \biggl| \int_0^T w_s^{-1} dW^i_s  \biggr| > C \sqrt{N} 
 \biggr) & \leq \frac1{c} \exp \bigl( - \frac{C^2 N}{c} \bigr),
\end{align*}
and then
\begin{align*} 
\mathbb{P} \biggl( \max_{i \in \{ 1,...,N \}}   \biggl| \int_0^T w_s^{-1} dW^i_s  \biggr| > C \sqrt{N} 
 \biggr) & \leq \frac{N}{c} \exp \bigl( - \frac{C^2 N}{c} \bigr).
\end{align*}
Hence,  
for any $p \geq 1$   
\begin{align*} 
\exp( C p N) \, \mathbb{P} \biggl( \max_{i \in \{ 1,...,N \}} \biggl| \int_0^T w_s^{-1} dW^i_s  \biggr| > C \sqrt{N} 
 \biggr) &\leq \frac{N}{c} \exp\bigg( CN \bigg( p - \frac{C}{c} \bigg) \bigg).
\end{align*}
Choosing $C$ large enough, we deduce from 
Doob's inequality that 
\begin{equation*}
\lim_{N \rightarrow +\infty}
{\mathbb E}
\Bigl[ \sup_{t \in [0,T]} \vert \triangle^N_{t} \vert^2
\Bigr] =0.
\end{equation*}

\textit{Third step.}
From inequality (\ref{vbound}) above, it is sufficient to show that $\tilde{g}^N(\tilde \mu^N_T) -v^N_T  \geq 0$
on the event $A_{N}$.

We thus return to the conclusion of the first step and we notice that, on the event $A_N$, for all $i=1,2,...,N$,
$$  |\tilde \mu^N_T - \tilde \mu^{i,N}_T|  \leq \frac{2}{N} \Bigg[ k_T + \sigma \max_{i \in \{ 1,...,N \}} \bigg| \int_0^T w_s^{-1} dW^i_s \bigg|  \Bigg] \leq \frac{c C}{\sqrt{N}},$$
for a new value of the constant $c$. 

Recall now from 
\eqref{eq:tilde:g} that $\gamma_N \in (0, r_\delta/2)$. 
\vspace{4pt}

\textit{Case a.}
Suppose that $ \tilde \mu^N_T < r_\delta - \frac{cC}{\sqrt{N}}$. Then, $\tilde \mu^{i,N}_T  < r_\delta $ (for all $i \in \{1,\cdots,N\}$) and by concavity of $g$ for values less than or equal to $r_\delta$, we obtain
$$ v^N_T = \frac{1}{N} \sum_{i=1}^{N} g(\tilde \mu^{i,N}_T) \leq g(\tilde \mu^N_T) \leq \tilde{g}^N(\tilde \mu^N_T).$$
\vspace{4pt}

\textit{Case b.}
Suppose that $ \tilde \mu^N_T > r_\delta + \frac{cC}{\sqrt{N}}$. Then, $\tilde \mu^{i,N}_T  > r_\delta$ (for all $i \in \{1,\cdots,N\}$),
and we obtain
$$ v^N_T = \frac{1}{N} \sum_{i=1}^{N} g(\tilde \mu^{i,N}_T) = -1 = g(\tilde \mu^N_T) \leq \tilde{g}^N(\tilde \mu^N_T).$$
\vspace{4pt}

\textit{Case c.}
Suppose that $ r_\delta - \frac{cC}{\sqrt{N}} \leq \tilde \mu^N_T \leq r_\delta + \frac{cC}{\sqrt{N}}$.
Then, $\tilde \mu^N_T - \frac{2cC}{\sqrt{N}} \leq \tilde \mu^{i,N}_T  \leq \tilde \mu^N_T + \frac{2cC}{\sqrt{N}}$
(for all $i \in \{1,\cdots,N\}$), since $g$ is non-increasing, we obtain, 
if $\gamma_{N} \geq 2cC/\sqrt{N}$,
\begin{equation*}
\begin{split}
0 = 
\tilde{g}^N(\tilde \mu^N_T) 
- g \bigl(\tilde\mu^N_{T} - \gamma_{N}\bigr)
&\leq 
\tilde{g}^N(\tilde \mu^N_T) 
- g\bigg( \tilde \mu^N_T - \frac{2cC}{\sqrt{N}} \bigg) 
\\
&\leq \tilde{g}^N(\tilde \mu^N_T) - \frac{1}{N} \sum_{i=1}^{N} g(\tilde \mu^{i,N}_T) = \tilde{g}^N(\tilde \mu^N_T) - v^N_T. 
\end{split}
\end{equation*}
The proof is complete.
\end{proof}


\subsubsection*{Comparison between $\tilde \theta^N$ and $\theta^N$}

Similar to Lemma \ref{lem:non:decreasing}, the first point is to notice that $\tilde{\theta}^N$ is non-increasing and that $\theta^N$ is also non-increasing.

\begin{lem}
\label{le:1}
For any fixed $N \geq 1$, 
the functions $\theta^N$ and $\tilde\theta^N$ are non-increasing in the space argument. 
\end{lem}
We make use of the non-increasing property to get the following crucial estimate:
\begin{lem}
\label{le:2}
We can choose $(\gamma_{N})_{N \geq 1}$
in \eqref{eq:tilde:g} such that
 $$\gamma_{N} \rightarrow 0, \quad {\gamma_{N}} \sqrt{N} \rightarrow +\infty,
 \quad \text{as} \quad N \rightarrow +\infty, $$ 
and then find a sequence $(\ell_{N})_{N \geq 1}$
such that 
$$\ell_{N} \rightarrow +\infty, \quad \ell_{N} \vert \ln(N) \vert^{-1/8}=0
 \quad \text{as} \quad N \rightarrow +\infty,$$
 and so that, for any
non-negative non-decreasing curve $\psi \in {\mathcal C}([0,T];\RR)$, 
which is strictly above the curve $t \mapsto   (r_{\delta}-r_{t})_{+}$  
on a left-open interval containing $[\delta,T]$,
it holds that
\begin{equation*}
\lim_{N \rightarrow + \infty}
\sup_{(t,x) \in [0,T], \vert x \vert \geq \ell_{N}/N + \psi_{t}}
\vert (\tilde \theta^{N} - \theta)(t,x) \vert =0.
\end{equation*}
\end{lem}

\begin{proof}
Take $\psi$ as in the statement. then, it is worth noticing that, for the prescribed values of 
$(t,x)$ in the supremum, $\theta(t,x)=-\textrm{\rm sign}(x)$.

Without any loss of generality, we can reduce the supremum to positive $x$'s. Hence, 
it suffices to prove that 
\begin{equation*}
\lim_{N \rightarrow  + \infty}
\sup_{(t,x) \in [0,T],  x  \geq \ell_{N}/N + \psi_{t}}
\vert 1 +\tilde \theta^{N}(t,x) \vert =0.
\end{equation*}
For a sequence $(\ell_{N})_{N \geq 1}$
such that 
$$\ell_{N} \rightarrow +\infty, \quad \ell_{N} \vert \ln(N) \vert^{-1/8}=0
 \quad \text{as} \quad N \rightarrow +\infty,$$
choose $(\gamma_{N}=\ell_{N}^{1/4} N^{-1/2})_{N \geq 1}$ 
in
\eqref{eq:tilde:g}.
By (\ref{eq:2612:1}), we know that, 
for any $t \in [0,T]$,
\begin{equation*}
- \int_{\ell_{N}^{1/2} N^{-1}+\psi_{t}}^{\ell_{N} N^{-1}+\psi_{t}} \tilde{\theta}^N(t,x) dx \geq - \int_{\ell_{N}^{1/2}N^{-1}+\psi_{t}}^{\ell_{N} N^{-1}+\psi_{t}} \theta^N(t,x) dx - \frac{2\ell_{N}^{1/2}}{r_{\delta} N}.
\end{equation*}
Hence, by Lemma \ref{le:1},
for any $t \in [0,T]$ and any $x \geq 
\ell_{N} N^{-1}+\psi_{t}$, 
\begin{equation}
\label{eq:thetaN:b-a}
- \tilde{\theta}^N(t,x) \geq 
- \tilde{\theta}^N\bigl(t,\ell_{N} N^{-1}+\psi_{t}\bigr) \geq - \theta^N\bigl(t,\ell_{N}^{1/2} N^{-1}+\psi_{t}\bigr) - \frac{2\ell_N^{1/2}}{r_{\delta} (\ell_{N} - \ell_{N}^{1/2})}. 
\end{equation}
We now recall Proposition \ref{propPSI}, from which we deduce
\begin{equation*}
\lim_{N \rightarrow +\infty} \sup_{t \in [0,T]} \bigl\vert 1 
+ \theta^N\bigl(t,\ell_{N}^{1/2} N^{-1}+\psi_{t}\bigr) \bigr\vert =0.
\end{equation*}
This completes the proof.
\end{proof}

\subsection{Proof of Theorem 
\ref{mainN} }

\begin{proof}
It suffices to invoke 
Theorem 
\ref{thm:general} (indexing the sequence of measures by $N$ instead of $\sigma_{0}$). 
(A1) is a consequence of 
Lemma \ref{le:2}. 
(A2) follows from 
the system 
\eqref{fbsdeNew:approx} and Lemma
\ref{le:0}.
(A3) 
 is a consequence of uniqueness (in law) 
to 
\eqref{fbsdeNext}, 
noticing that
$((-X^i_{t},-V_{t}^i,z_{t}^{i,k})_{t \in [0,T]})_{i,k=1,\cdots,N}$
is a solution of the system
\eqref{fbsdeNext}
driven by $(-W^1_{t},\cdots,-W^N_{t})_{t \in [0,T]}$.
\end{proof}

\section{Appendix: Proof of
Proposition
\ref{propPSI}}
\label{se:appendix}

Consider $\sigma_0 \in (0,1)$ and $(t,x) \in [0,\delta) \times\mathbb{R}$ with $\vert x \vert \leq r_{t} - r_{\delta}$. 
It is easily checked (using a change of variable) that $\theta^{\sigma_{0}}$ is odd in $x$ and, thus, 
that $\Psi$ is also odd in $x$. Therefore, we can just focus on $\Psi(t,x,\sigma_{0})$ for $x >0$. 

To simplify notation, we write $\lambda := \sigma_0^{-2} \in (1 , +\infty)$, so that obtaining asymptotic expressions as $\sigma_0 \rightarrow 0$ is equivalent to obtaining asymptotic expressions as $\lambda \rightarrow +\infty$. 
We also use the definitions:
\begin{equation*}
\begin{split}
&g(y) :=  -\frac{y}{r_\delta}  {\mathbf 1}_{ |y| \leq r_\delta } - \textrm{\rm sign}(y) {\mathbf 1}_{ |y| > r_\delta }, 
\quad h(y):=  -\int_0^y g(v)dv - \frac{(x-y)^2}{2 r_t}, 
\end{split}
\end{equation*}
and
\begin{equation*}
\textrm{\rm erf}(y) := \frac{2}{\sqrt{\pi}}\int_0^y \exp\bigl(- v^2\bigr)dv, 
\end{equation*}
for $(t,y) \in [0,T) \times \RR$. 

Lastly, throughout the proof, we use the generic notation $Q(y)$ for a polynomial function of $y$ of degree less than or equal to $1$. Possibly, $Q$ may depend on $(t,x)$.
\vspace{5pt}

\textit{Preliminary computation:} For every $(t,x)$ as above, we define $\bar{y} := \frac{-xr_\delta}{r_t - r_\delta}$, $y^*_1 := x-r_t$,  $y^*_2 := x + r_t.$ It holds that $y^*_1$ is a global maximum of $h$ on $(-\infty,-\bar{y})$ and $y^*_2$ is a global maximum of $h$ on $(\bar{y},+\infty).$ Indeed, 
\begin{equation*}
\label{eq:h':case:2}
\begin{split}
h'(y) &= -1 + \frac{x-y}{r_t} > -1 + \frac{x-y^*_1}{r_t} = 0,\hspace{2mm} \textrm{\rm for all} \hspace{2mm} y \in (-\infty,y^*_1),  \\
h'(y)& = -1 + \frac{x-y}{r_t}  < -1 + \frac{x-y^*_1}{r_t} = 0 , \hspace{2mm}\textrm{\rm for all} \hspace{2mm} y \in ( y^*_1,-r_\delta), 
\\
h'(y) &= y \bigg( \frac{r_t -r_\delta}{r_t r_\delta} \bigg) + \frac{x}{r_t}  < \bar{y} \bigg( \frac{r_t -r_\delta}{r_t r_\delta} \bigg) + \frac{x}{r_t} = 0 , \hspace{2mm} \textrm{\rm for all} \hspace{2mm} y \in (-r_\delta, \bar{y}),  
\\
h'(y) &= y \bigg( \frac{r_t -r_\delta}{r_t r_\delta} \bigg) + \frac{x}{r_t}  > \bar{y} \bigg( \frac{r_t -r_\delta}{r_t r_\delta} \bigg) + \frac{r_\delta-r_t}{r_t} = 0 , \hspace{2mm} \textrm{\rm for all} \hspace{2mm} y \in ( \bar{y},r_\delta),  
\\
h'(y)& = 1 + \frac{x-y}{r_t}  > 1 + \frac{x-y^*_2}{r_t} = 0 , \hspace{2mm}\textrm{\rm for all} \hspace{2mm} y \in ( r_\delta,y^*_2), 
\\
h'(y) &= 1 + \frac{x-y}{r_t} < 1 + \frac{x-y^*_2}{r_t} = 0,\hspace{2mm} \textrm{\rm for all} \hspace{2mm} y \in (y^*_2,+\infty).  \end{split}
\end{equation*}
Now, for a polynomial function of order less than or equal to $1$, we compute
\begin{align*}
&\int_{-\infty}^{\infty} Q(y)\exp(\lambda h(y)) dy \notag 
\\
&= \int_{-\infty}^{-r_{\delta}} Q(y)
\exp(\lambda h(y)) dy 
+ \int_{-r_{\delta}}^{r_{\delta}} Q(y) 
\exp(\lambda h(y)) dy 
+ \int_{r_{\delta}}^{+\infty} Q(y)\exp(\lambda h(y)) dy \notag 
\\
&=  \int_{-\infty}^{-r_{\delta}-y^*_1} Q(u+y^*_1)
\exp(\lambda h(u+y^*_1)) du 
+\int_{-r_{\delta}}^{r_{\delta}} Q(y) 
\exp(\lambda h(y)) dy 
\\
&\hspace{15pt} + \int_{r_{\delta}-y^*_2}^{+\infty} Q(u+y^*_2)\exp(\lambda h(u+y^*_2)) du.
\end{align*}
We have
\begin{equation*}
\begin{split}
&\forall  u \in ( -\infty,-r_{\delta}-y^*_1), \hspace{2mm} h(u+y^*_1) - h(y^*_1) =
- \int_{u+y_{1}^*}^{y_{1}^*}
\frac{y_{1}^*-z}{r_{t}}
dz
=
 -  \frac{u^2}{2r_{t}} < 0, 
\\
&\forall  u \in ( r_{\delta}-y^*_2,+\infty), \hspace{2mm} h(u+y^*_2) - h(y^*_2) =
\int_{y_{2}^*}^{u+y_{2}^*}
\frac{y_{2}^*-z}{r_{t}}
dz = 
 - \frac{u^2}{2r_{t}} < 0.
\end{split}
\end{equation*}
Now, letting
\begin{equation*}
B_{1}: = - \frac{r_{\delta} + y_{1}^*}{\sqrt{2r_{t}}} 
= - \frac{x + r_{\delta}  -r_{t}}{\sqrt{2r_{t}}}
= \frac{r_{t} - r_{\delta}  -x}{\sqrt{2r_{t}}}
\geq 0,
\end{equation*}
we have
\begin{equation*}
\begin{split}
&\int_{-\infty}^{-r_{\delta}-y_{1}^*} Q(u+y^*_1)
\exp(\lambda h(u+y^*_1)) du   
\\
&= \exp \bigl(\lambda h(y^*_{1})\bigr) \int_{-\infty}^{-r_{\delta}-y_{1}^*} \bigl( Q(y^*_1)
+ Q'(y_{1}^*) u \bigr) 
\exp( - \frac{\lambda u^2}{2 r_{t}}) du  
\\
&= \sqrt{2r_{t}} \exp \bigl(\lambda h(y^*_{1})\bigr) \int_{-\infty}^{B_{1}} \bigl( Q(y^*_1)
+  \sqrt{2 r_{t}} Q'(y_{1}^*) s \bigr) 
\exp( - \lambda s^2) ds
\\
&= \sqrt{2r_{t}} \exp \bigl(\lambda h(y^*_{1})\bigr)
\biggl( Q(y_{1}^*) \sqrt{\frac{\pi}{4}} \frac{\text{erf}(\sqrt{\lambda} B_{1})+1}{\sqrt{\lambda}}
-
Q'(y_{1}^*)\frac{\sqrt{2r_{t}}}{2\lambda}
 \exp( - \lambda B_{1}^2)  \biggr)
 \\
 &= \exp \bigl(\lambda h(y^*_{1})\bigr)
 Q(y_{1}^*) \sqrt{\frac{\pi r_{t}}{2}} \frac{\text{erf}(\sqrt{\lambda} B_{1})+1}{\sqrt{\lambda}}
-
\exp \bigl(\lambda h(y^*_{1})\bigr)
Q'(y_{1}^*)
\frac{r_{t}}{\lambda}
 \exp( - \lambda B_{1}^2). 
\end{split}
\end{equation*}
Similarly, letting 
letting
\begin{equation*}
B_{2} :=  \frac{r_{\delta} - y_{2}^*}{\sqrt{2r_{t}}}
=
\frac{r_{\delta} - r_{t} - x}{\sqrt{2r_{t}}}
=
-
\frac{r_{t} - r_{\delta}  + x}{\sqrt{2r_{t}}}
 \leq 0,
\end{equation*}
we have
\begin{equation*}
\begin{split}
&\int_{r_{\delta}-y_{2}^*}^{+\infty} Q(u+y^*_2)
\exp(\lambda h(u+y^*_2)) du   
\\
&= \exp \bigl(\lambda h(y^*_{2})\bigr)
 Q(y_{2}^*) \sqrt{\frac{\pi r_{t}}{2}} \frac{1-\text{erf}(\sqrt{\lambda} B_{2})}{\sqrt{\lambda}}
+
\exp \bigl(\lambda h(y^*_{2})\bigr)
Q'(y_{2}^*)
\frac{r_{t}}{\lambda}
 \exp( - \lambda B_{2}^2). 
\end{split}
\end{equation*}
Therefore, 
\begin{equation}
\label{eq:case2:expansion}
\begin{split}
\int_{-\infty}^{\infty}  Q(y)\exp(\lambda h(y)) dy &= \exp(\lambda h(y^*_1)) Q(y^*_1)\sqrt{\frac{\pi  r_t}{2}} \frac{ \big(\text{erf}(\sqrt{ \lambda } B_1)+1\big)}{\sqrt{\lambda}} 
\\ 
&\hspace{15pt} + \exp(\lambda h(y^*_2)) Q(y^*_2) \sqrt{\frac{ \pi r_t }{2}}  \frac{ \big(1-\text{erf}(\sqrt{ \lambda} B_2)\big)}{\sqrt{\lambda}}
+ {\mathcal R}(t,x,\lambda),
\end{split}
\end{equation}
with
\begin{align*}
{\mathcal R}(t,x,\lambda) &= - \exp(\lambda h(y^*_1)) Q'(y^*_1) r_t 
 \frac{\exp(-\lambda B_1^2)}{\lambda} 
 \notag \\ 
& \hspace{15pt} + \exp(\lambda h(y^*_2)) Q'(y^*_2) r_{t} \frac{\exp(-\lambda B_2^2)}{\lambda}
+ 
\int_{-r_{\delta}}^{r_{\delta}} Q(y) 
\exp(\lambda h(y)) dy.
\end{align*}  

\textit{Using Cole-Hopf formula:}
Define now $q(y) = \frac{x-y}{r_t} + \textrm{\rm sign}(x)$, and recall that
$$ \Psi(t,x,\sigma_0) = \theta^{\sigma_0}(t,x) - \theta(t,x) = \frac{\int_{-\infty}^{+\infty} q(y) \exp(\lambda h(y)) dy}{\int_{-\infty}^{+\infty} \exp(\lambda h(y)) dy} .$$ 
\vspace{5pt}

Recall that $0 \leq x < r_{t} - r_{\delta}$. Then, $q(y)= \frac{x-y}{r_t} + 1$, and then $q(y_{1}^*)=2$ and 
$q(y_{2}^*)=0$. Therefore,
by \eqref{eq:case2:expansion} with $Q(y)=q(y)$ and with $Q(y)=1$, we obtain
$$ \vert \Psi(t,x,\sigma_0)\vert  \leq  \frac{2 \sqrt{\pi r_t} \exp(\lambda h(y^*_1)) \bigg( \frac{1+\textrm{\rm erf}(\sqrt{\lambda}B_1)}{\sqrt{\lambda}}\bigg) +\sqrt{2}  \vert {\mathcal R}(t,x,\lambda) \vert}{\sqrt{\pi r_t} \exp(\lambda h(y^*_1)) \bigg( \frac{1+\textrm{\rm erf}(\sqrt{\lambda}B_1)}{\sqrt{\lambda}}\bigg) + \sqrt{\pi r_t} \exp(\lambda h(y^*_2)) \bigg( \frac{1-\textrm{\rm erf}(\sqrt{\lambda}B_2)}{\sqrt{\lambda}}\bigg)},$$
the remainder ${\mathcal R}$  being computed with $Q=q$. Therefore,

\begin{align*}
|\Psi(t,x,\sigma_0) |& \leq \frac{2\sqrt{\pi r_t} \exp(\lambda h(y^*_1)) \bigg( \frac{1+\textrm{\rm erf}(\sqrt{\lambda}B_1)}{\sqrt{\lambda}}\bigg)}{ \sqrt{\pi r_t} \exp(\lambda h(y^*_2)) \bigg( \frac{1-\textrm{\rm erf}(\sqrt{\lambda}B_2)}{\sqrt{\lambda}}\bigg)}
+\sqrt{2} {\mathcal I}_{1}(t,x,\lambda) +\sqrt{2} {\mathcal I}_{2}(t,x,\lambda) \notag
\\
&= 2 \exp\big( \lambda [h(y^*_1) - h(y^*_2)] \big) \bigg( \frac{1 + \textrm{\rm erf}(\sqrt{\lambda} B_1)}{1 - \textrm{\rm erf}(\sqrt{\lambda} B_2)}\bigg)
+ \sqrt{2}{\mathcal I}_{1}(t,x,\lambda) + \sqrt{2}{\mathcal I}_{2}(t,x,\lambda),
\end{align*}
where
\begin{equation}
\label{eq:I1:I2}
\begin{split}
{\mathcal I}_{1}(t,x,\lambda) &=  \frac{\exp(\lambda h(y^*_1))  
 \frac{\exp(-\lambda B_1^2)}{\lambda} 
  + \exp(\lambda h(y^*_2))  \frac{\exp(-\lambda B_2^2)}{\lambda}}{
   \sqrt{\pi r_t} \exp(\lambda h(y^*_2)) \bigg( \frac{1-\textrm{\rm erf}(\sqrt{\lambda}B_2)}{\sqrt{\lambda}}\bigg)},
   \\
   {\mathcal I}_{2}(t,x,\lambda) &=  \frac{ \int_{-r_{\delta}}^{r_{\delta}} \vert q(y) \vert \exp( \lambda h(y) ) dy}{ \sqrt{\pi r_t} \exp(\lambda h(y^*_2)) \bigg( \frac{1-\textrm{\rm erf}(\sqrt{\lambda}B_2)}{\sqrt{\lambda}}\bigg)}.
\end{split}
\end{equation}
Notice now the following key facts:
\begin{equation}
\label{eq:B22}
\begin{split}
&h(y^*_2) - h(y^*_1) = 
\int_{y_{1}^*}^{-r_{\delta}}
(-1) dz +
\int_{-r_{\delta}}^{r_{\delta}}
\frac{z}{r_{\delta}} dz
+ \int_{r_{\delta}}^{y_{2}^*}
dz - \frac{(x-y_{2}^*)^2 - (x-y_{1}^*)^2}{2 r_{t}}
\\
&\hspace{64pt} = \bigl( r_{\delta} + y_{1}^* \bigr) + 0 + \bigl( y_{2}^* - r_{\delta}
\bigr) - \frac{r_{t}^2 - r_{t}^2}{2r_{t}}
=  2x,
\\
&B_{2} <0, \quad B_{2}^2 = \frac{\bigl( x+r_{t}-r_{\delta}\bigr)^2}{2 r_{t}}  \geq \max 
\Bigl( \frac{x^2}{2r_{t}}, \frac{\bigl( r_{t}-r_{\delta}\bigr)^2}{2r_{t}}\Bigr). 
\end{split}
\end{equation}
In particular, 
\begin{equation*}
2 \exp\big( \lambda [h(y^*_1) - h(y^*_2)] \big) \bigg( \frac{1 + \textrm{\rm erf}(\sqrt{\lambda} B_1)}{1 - \textrm{\rm erf}(\sqrt{\lambda} B_2)}\bigg)
\leq 4 \exp \bigl( - 2 \lambda x \bigr),
\end{equation*}
and 
\begin{equation*}
\vert {\mathcal I}_{1}(t,x,\lambda) \vert \leq \frac2{\sqrt{\lambda \pi  r_{t}}}.  
\end{equation*}

\textit{Handling ${\mathcal I}_{2}$:}
To handle ${\mathcal I}_{2}(t,x,\lambda)$, we notice that $\vert q(y) \vert \leq 1$ for $y \in (-r_{\delta},r_{\delta})$. Also, 
for $y \in (\bar y,r_{\delta})$
\begin{equation*}
\begin{split}
h(r_{\delta}) - h(y) 
&= \int_y^{r_{\delta}} h'(z) dz
= \int_y^{r_{\delta}} \bigl( z - \bar y \bigr)  \frac{r_t -r_\delta}{r_t r_\delta} 
dz
\geq \frac{r_t -r_\delta}{2 r_{t} r_{\delta}} (r_{\delta}- y)^2.
\end{split}
\end{equation*}
Hence,
\begin{equation*}
\begin{split}
\int_{\bar y}^{r_{\delta}} \vert q(y) \vert \exp \bigl( \lambda h(y) \bigr) dy 
&\leq  \exp \bigl( \lambda h(r_{\delta}) \bigr) \int_{\bar y}^{r_{\delta}} \exp \bigl( -
\lambda \frac{(r_{t}-r_{\delta})(r_{\delta}-y)^2}{2r_{t} r_{\delta}}\bigr) dy
\\
&\leq \exp \bigl( \lambda h(r_{\delta})\bigr) \sqrt{\frac{2 \pi r_{t}
r_{\delta}}{\lambda (r_{t} - r_{\delta})}}.
\end{split}
\end{equation*}
By the same argument, 
for $y \in (-r_{\delta},\bar y)$, 
\begin{equation*}
h(y) - h(-r_{\delta}) = \int_{-r_{\delta}}^y 
h'(z) dz =
\int_{-r_{\delta}}^y \bigl( z - \bar y \bigr)  \frac{r_t -r_\delta}{r_t r_\delta} 
dz \leq 
-
\frac{r_t -r_\delta}{2 r_{t} r_{\delta}} (r_{\delta}+ y)^2,
\end{equation*}
and using in addition the fact that $h$ is decreasing on $(y_{1}^*,\bar y)$, 
\begin{equation*}
\begin{split}
\int_{-r_{\delta}}^{\bar y} \vert q(y) \vert \exp \bigl( \lambda h(y) \bigr) dy 
\leq 
\exp \bigl( \lambda h(-r_{\delta})\bigr) 
 \sqrt{\frac{2 \pi r_{t}
r_{\delta}}{\lambda (r_{t} - r_{\delta})}}
\leq \exp \bigl( \lambda h(y_{1}^*)\bigr) 
 \sqrt{\frac{2 \pi r_{t}
r_{\delta}}{\lambda (r_{t} - r_{\delta})}}.
\end{split}
\end{equation*}
Now, 
\begin{equation*}
h(y^*_1) = h(y^*_2)  -2x, \quad h(r_{\delta}) = h(y_{2}^*)
+ \int_{y_{2}^*}^{r_{\delta}} \frac{ y_{2}^*-z}{r_{t}}
dz
=h(y_{2}^*)
 - \frac{1}{2r_{t}}{(r_{\delta} - y_{2}^*)^2}
= h(y_{2}^*) - B_{2}^2.
\end{equation*}
Hence,
\begin{equation*}
\int_{-r_{\delta}}^{r_{\delta}} \vert q(y) \vert \exp \bigl( \lambda h(y) \bigr) dy 
\leq \sqrt{\frac{2 \pi r_{t}r_{\delta}}{\lambda (r_{t} - r_{\delta})}} 
\exp \bigl( \lambda h(y_{2}^*) \bigr) \Bigl( \exp\bigl(-2 \lambda x \bigr) + \exp \bigl( -
\lambda B_{2}^2 \bigr) 
\Bigr).
\end{equation*}
Therefore, by 
\eqref{eq:I1:I2},
\begin{equation*}
\begin{split}
\vert {\mathcal I}_{2}(t,x,\lambda) \vert &\leq 
 \sqrt{\frac{2  
r_{\delta}}{(r_{t} - r_{\delta})}}
\Bigl( \exp\bigl(-2 \lambda x \bigr) + \exp \bigl( - \lambda B_{2}^2 \bigr) 
\Bigr).
\end{split}
\end{equation*}

\textit{Conclusion:}
Collecting the various terms, we obtain, for $x>0$, 
\begin{equation*}
\vert \Psi(t,x,\sigma_{0}) 
\vert \leq 
\Bigl( 4
+ 2 \sqrt{\frac{ 
r_{\delta}}{(r_{t} - r_{\delta})}}
\Bigr)
 \exp (-2 \lambda x )
 + 
  \frac{2 \sqrt{2}}{\sqrt{\lambda \pi  r_{t}}}
  +
2   \sqrt{\frac{r_{\delta}}{(r_{t} - r_{\delta})}}
\exp 
\Bigl( - \lambda
\frac{\bigl( r_{t}-r_{\delta}\bigr)^2}{2r_{t}}
\Bigr).  
\end{equation*}
By symmetry, we get the same result for $x<0$.

\section*{Acknowledgment}

Fran\c{c}ois Delarue and Rinel Foguen Tchuendom are partially supported by  
ANR MFG  (ANR-16-CE40-0015-01). Fran\c{c}ois Delarue is also partially supported
by Institut Universitaire de France.  

\bibliographystyle{plain}
\bibliography{main-rinel}
\vspace{5pt}

{\small {\sc Fran\c{c}ois DELARUE

Rinel FOGUEN TCHUENDOM

Laboratoire J.-A. Dieudonn\'e,

Universit\'e de Nice Sophia-Antipolis and UMR CNRS 7351, 

Parc Valrose, 06108 Nice Cedex 02, France}

\texttt{delarue@unice.fr}, 

\texttt{Rinel.Foguen$\_$tchuendom@unice.fr}}

\end{document}